\documentclass[11pt,a4paper,twoside]{amsart}
\usepackage[top=1in, bottom=1.25in, marginparwidth=1in, inner=1in, outer=1in]{geometry}

\usepackage{lmodern}
\usepackage[T1]{fontenc}
\usepackage[utf8]{inputenc}

\usepackage{tikz}
\usepackage{tikz-cd}
\usetikzlibrary{decorations.pathreplacing}
\usetikzlibrary{calc}
\colorlet{darkgreen}{green!70!black}

\usepackage{caption}
\usepackage[labelfont=up,margin=5pt]{subcaption}
\usepackage[section]{placeins}

\usepackage{setspace} 
\usepackage[bookmarks=true,pdfencoding=auto,colorlinks=true, urlcolor=black, linkcolor=black, citecolor=black]{hyperref}

\usepackage{amsmath}
\usepackage{amsfonts}
\usepackage{amssymb}
\usepackage{amsthm}
\usepackage{mathtools}
\mathtoolsset{mathic=true}
\usepackage{bm}

\newtheorem{theorem}{Theorem}[section]
\newtheorem*{main}{Main Theorem}
\newtheorem{lemma}[theorem]{Lemma}
\newtheorem{conjecture}[theorem]{Conjecture}

\newtheorem{questions}[theorem]{Questions}
\newtheorem{corollary}[theorem]{Corollary}
\newtheorem{proposition}[theorem]{Proposition}
\theoremstyle{definition}
\newtheorem{definition}[theorem]{Definition}
\newtheorem{example}[theorem]{Example}
\newtheorem{observation}[theorem]{Observation}
\newtheorem{remark}[theorem]{Remark}

\usepackage{enumerate}

\title{L-space knots have no essential Conway spheres}
\date{}

\author{Tye Lidman}
\address{Department of Mathematics \\ North Carolina State University}
\email{tlid@math.ncsu.edu}

\author{Allison H. Moore}
\address{Department of Mathematics \& Applied Mathematics \\ Virginia Commonwealth University}
\email{moorea14@vcu.edu}

\author{Claudius Zibrowius}
\address{Fakultät für Mathematik \\ Universität Regensburg}
\email{claudius.zibrowius@posteo.net}

\newcommand{\vc}[1]{\vcenter{\hbox{#1}}}%
\newcommand{\mypic}[2]{%
  \newcommand{#2}{%
    \vc{%
      \includegraphics[page=#1]%
      {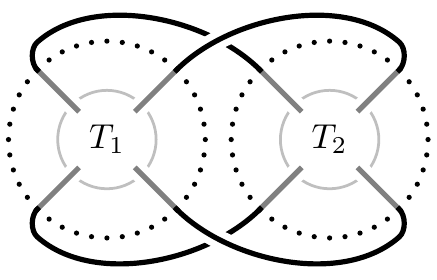}%
    }%
  }%
}%

\mypic{1}{\tanglepairingI}
\mypic{2}{\tanglepairingII}
\mypic{3}{\rigidcurveIntro}
\mypic{4}{\pretzeltangleDownstairs}
\mypic{5}{\pretzeltangleUpstairs}
\mypic{6}{\pretzeltangle}
\mypic{7}{\InlineTrivialTangle}
\mypic{8}{\xI}
\mypic{9}{\xII}
\mypic{10}{\xIII}
\mypic{11}{\xyI}
\mypic{12}{\xyII}
\mypic{13}{\xyIII}
\mypic{14}{\xyIV}
\mypic{15}{\PairingIrrRat}
\mypic{16}{\PairingIrrIrrParallel}
\mypic{17}{\PairingIrrIrrNonparallel}
\mypic{18}{\AlexLemma}
\mypic{19}{\HDstabilizeBspt}
\mypic{20}{\HDstabilizeTrivial}
\mypic{21}{\splittangle}
\mypic{22}{\bigon}
\mypic{23}{\bigonlift}

\DeclareMathOperator{\Alex}{\hat{A}}
\newcommand{\GeneralAlex}{\mathfrak{A}}

\DeclareMathOperator{\CFTd}{CFT^\partial}

\DeclareMathOperator{\HFT}{HFT}
\newcommand{\Rat}{\mathfrak{r}}

\DeclareMathOperator{\mr}{mr}

\DeclareMathOperator{\HFKhat}{\widehat{HFK}}
\DeclareMathOperator{\HFL}{\widehat{HFL}}

\DeclareMathOperator{\HFhat}{\widehat{HF}}
\DeclareMathOperator{\HF}{HF}
\DeclareMathOperator{\HFr}{\widetilde{HF}}
\DeclareMathOperator{\CFK}{CFK}
\DeclareMathOperator{\CFhat}{\widehat{CF}}
\DeclareMathOperator{\SFH}{SFH}

\newcommand{\Ad}{\operatorname{\mathcal{A}}^\partial}
\DeclareMathOperator{\Gen}{\mathcal{G}}

\newcommand{\field}{\mathbb{F}}

\DeclareMathOperator{\SL}{SL}
\DeclareMathOperator{\im}{im}
\DeclareMathOperator{\id}{id}

\newcommand{\FourPuncturedSphere}{S^2_4}
\newcommand{\PuncturedPlane}{\mathbb{R}^2\smallsetminus \mathbb{Z}^2}
\newcommand{\QPI}{\operatorname{\mathbb{Q}P}^1}

\newcommand{\LiftG}{\tilde{\gamma}}
\newcommand{\LiftT}{\tilde{\vartheta}}
\newcommand{\LiftX}{\tilde{x}}
\newcommand{\ParaCovering}{P}
\newcommand{\Rational}{\mathbf{r}}
\newcommand{\Special}{\mathbf{s}}
\newcommand{\Gammai}{\textcolor{red}{\Gamma_1}}
\newcommand{\Gammaii}{\textcolor{blue}{\Gamma_2}}
\newcommand{\CC}[1]{\textcolor{red}{C_{#1}}}
\newcommand{\DD}[1]{\textcolor{blue}{D_{#1}}}

\DeclareMathOperator{\TEo}{\textnormal{\texttt{o}}}
\DeclareMathOperator{\TEi}{\textnormal{\texttt{i}}}
\DeclareMathOperator{\TEj}{\textnormal{\texttt{j}}}
\DeclareMathOperator{\TEk}{\textnormal{\texttt{k}}}
\DeclareMathOperator{\TEl}{\textnormal{\texttt{l}}}

\DeclareMathOperator{\TEI}{\textnormal{\texttt{1}}}
\DeclareMathOperator{\TEII}{\textnormal{\texttt{2}}}
\DeclareMathOperator{\TEIII}{\textnormal{\texttt{3}}}
\DeclareMathOperator{\TEIV}{\textnormal{\texttt{4}}}

\def\F{\mathbb{F}}

\def\Q{\mathbb{Q}}
\def\R{\mathbb{R}}
\def\Z{\mathbb{Z}}

\def\co{\colon\thinspace\relax}

\newcommand{\Contiguous}[1]{\mathcal{C}_{#1}}

\newcommand{\tikzscale}{0.85}
\begin{document}

\begin{abstract}
We prove that L-space knots do not have essential Conway spheres with the technology of peculiar modules, a Floer theoretic invariant for tangles.
\end{abstract}

\maketitle
\section{Introduction}\label{sec:intro}

We consider the problem of whether Dehn surgery along a knot $K$ in $S^3$ produces an L-space.  L-spaces are closed, oriented three-manifolds with the simplest possible Heegaard Floer homology.  This class includes, for example, all three-manifolds with finite fundamental group \cite[Proposition 2.3]{HFKlens}.  While Floer theory has become a very effective modern tool for answering questions about Dehn surgery, a geometric characterization of knots admitting L-space surgeries remains a difficult outstanding problem. 
Ozsváth and Szabó established a structure theorem for the knot Floer homology of such knots \cite{HFKlens}, which allows one to show that they are fibered \cite{Ghiggini, Ni} and prime \cite{Krcatovich, HeddenWatson, BaldwinVelaVick}.  
Both properties are inherently statements about surfaces in the knot exterior. 
It is natural to ask whether the existence of certain essential surfaces in the complement of a knot can obstruct non-trivial surgeries yielding L-spaces. Recall that a Conway sphere is a two-sphere intersecting the knot transversely in four points. It is \emph{essential} if the corresponding four-punctured sphere is incompressible in the knot exterior. 

The main purpose of this article is to prove:

\begin{main}
A knot in $S^3$ with a non-trivial L-space surgery admits no essential Conway sphere.  
\end{main}

This answers affirmatively the conjecture posed by the first and second authors in \cite{LidmanMoore}.
As an immediate corollary, we obtain:
\begin{corollary}
	Conway mutation preserves L-space knots.\qed
\end{corollary}

We now describe the rough strategy for the proof of this theorem. 
Following a construction of the third author, we associate with a four-ended tangle \(T\) a decorated immersed multicurve \(\HFT(T)\) in the four-punctured sphere. 
This multicurve is a geometric realization of a bordered sutured Heegaard Floer invariant called the peculiar module of $T$, and
the Lagrangian Floer homology of two such multicurves describes the link Floer homology of a suitable tangle sum of the two corresponding tangles \cite{PQMod}. 
Furthermore, a structure theorem has been established for what the individual components of multicurves can look like \cite{PQSym}. 
Each component is one of two types: rational (which is the immersed curve invariant of a rational tangle) or special.  
See Figure~\ref{fig:HFT:example} and Section~\ref{sec:review:HFT:geography} below.   
The invariant $\HFT$ in fact detects rational tangles \cite[Theorem~6.2]{PQMod}. 
Generalizing rational tangles, a {\em split} tangle is a rational tangle with links possibly tied into either of the strands. 
A first major step towards the Main Theorem is a strengthening of this detection result to split tangles, which may be of independent interest:
\begin{theorem}\label{thm:detection:split:Intro}
	\(\HFT(T)\) detects split tangles.
\end{theorem}  
See Theorem~\ref{thm:detection:split} for a more detailed statement of this result. 

Now, suppose that a knot \(K = T_1 \cup T_2\) is decomposed along a Conway sphere into two four-ended tangles \(T_1\) and \(T_2\) as shown in Figure~\ref{fig:tanglepairing}. 
Assuming that this Conway sphere is essential means that neither \(T_1\) nor \(T_2\) are split tangles. 
Then, using Theorem~\ref{thm:detection:split:Intro}, we show that \(\HFT(T_i)\) contains either special components or rational components of different slopes.  
As we will argue in Lemma \ref{lem:lspacepinch}, should an essential tangle decomposition \(K = T_1 \cup T_2\) exist, then a rational tangle replacement of either of $T_1$ or $T_2$ would yield an auxiliary knot \(T_1 \cup Q_{s_2}\) or \(Q_{s_1} \cup T_2\), respectively, which must also admit an L-space surgery. 
Applying such rational tangle replacements, we perform a case analysis based on the composition of rational and special components of \(\HFT(T_i)\) in order to analyze their Lagrangian Floer homology. 
In all cases we show the knot Floer homology of \(K\) violates Ozsváth and Szabó's structure theorem for the knot Floer homology of knots admitting non-trivial L-space surgeries. 

Our arguments rely on both the Alexander and \(\delta\)-grading structure of the knot Floer homology of knots admitting L-space surgeries.   Hence, we ask the following:
\begin{questions}
Is there a knot \(K\) admitting an essential Conway sphere such that for all Alexander gradings \(A\), \(\dim \HFKhat(K,A) \leq 1\)?  
In fact, is there any knot that admits no non-trivial L-space surgery and whose knot Floer homology satisfies this constraint?
\end{questions}

\begin{figure}[bt]
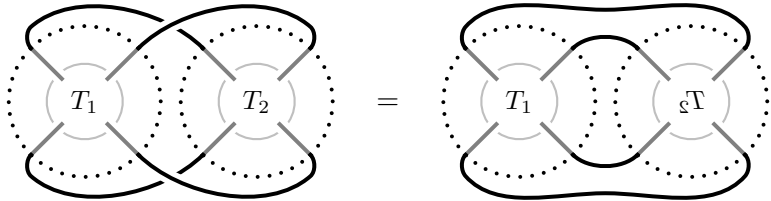

	\centering
	\(
	\tanglepairingI
	\quad = \quad
	\tanglepairingII
	\)
	\caption{Two diagrams defining the same tangle decomposition of the link \(T_1\cup T_2\). The tangle \protect\reflectbox{\(T_2\)} is the result of rotating \(T_2\) around the vertical axis. By rotating the entire link on the right-hand side around the vertical axis, we can see that \(T_1\cup T_2=T_2\cup T_1\).}
	\label{fig:tanglepairing} 
\end{figure}

The characterization of knots admitting cyclic, or more generally, elliptic or exceptional surgeries has been a problem of lasting interest and difficulty in three-manifold topology. Our Main Theorem recovers a result that is proved implicitly in the work of Wu \cite[p.~173]{Wu:DehnSurgeryArborescent}. Wu proved that if a knot $K$ is the union of two non-split tangles, then either $K$ is a $(2, q)$-cable of a composite knot, or the exterior of $K$ contains an essential lamination which remains essential after all non-trivial surgeries. In the former case, there is a torus which remains incompressible after every non-trivial surgery, which makes the fundamental group of surgery infinite. In the latter, the fundamental group is infinite by results of Gabai and Oertel \cite{GO}.  

Since all elliptic three-manifolds are L-spaces, the Main Theorem immediately gives a new proof of Wu's result:
\begin{corollary}\label{infinite}
Let \(K\) be a knot in \(S^3\) with an essential Conway sphere.  Then \(\pi_1(S^3_{p/q}(K))\) is infinite for all \(p/q \in \mathbb{Q}\).  \qed
\end{corollary}

Additionally we have: 

\begin{corollary}\label{cor:BDC}
Let \(K\) be a hyperbolic knot with a non-trivial L-space surgery.  The double cover of \(S^3\) branched over \(K\) is either hyperbolic or, up to orientation-reversal, the Seifert fibered space \(S^2(-1; \frac{1}{2}, \frac{1}{3}, \frac{1}{2n+1})\) for \(n \geq 3\).  
\end{corollary}

\begin{proof}
Let $\Sigma(K)$ denote the branched double cover of $K$. If $\Sigma(K)$ is not hyperbolic, then because $K$ has no essential Conway sphere, $\Sigma(K)$ is Seifert fibered, and hence $K$ is Montesinos \cite[Section 2]{Paoluzzi}. 
The hyperbolic Montesinos knots with L-space surgeries are, up to mirroring, the pretzel knots $P(-2,3,2n+1)$ with $n \geq 3$ \cite{BM}. 
The branched double covers of these knots are the Seifert fibered spaces described in the corollary.  
\end{proof}

Finally, we mention an open problem. A knot \(K\) has an essential $n$-string tangle decomposition if there is an embedded sphere that transversally intersects $K$ in $2n$ points and determines an essential surface in the knot exterior. If $K$ has no essential $n$-string tangle decomposition, then $K$ is called $n$-string prime. In this terminology, a prime knot is $1$-string prime. A knot without any essential Conway sphere is $2$-string prime. Baker and the second author conjecture \cite[Conjecture 19]{BM} more generally:

\begin{conjecture}\label{conj:n-string_prime}
	A knot which admits an L-space surgery is \(n\)-string prime for all \(n\geq1\).
\end{conjecture}

Baker and Motegi show that this conjecture is true for satellite knots with \(n\leq3\) \cite[Theorem~7.7]{BakerMotegi}. (The conjecture in their assumptions is the L-space gluing theorem, which was proved by Hanselman, Rasmussen and Watson~\cite{HRW}.) Motegi has pointed out the following application of our main theorem: 

\begin{corollary}\label{cor:satellites}
	If \(n\leq5\), Conjecture~\ref{conj:n-string_prime} holds for satellite knots. 
\end{corollary}

\begin{proof}
	Suppose a satellite knot $P(K)$ with pattern \(P\) and companion \(K\) admits an essential \(n\)-string tangle decomposition along some 2-sphere \(S\).
	Then \(S\) intersects \(P(K)\) transversely in \(2n\) points.
	By \cite[Theorem~7.4]{BakerMotegi}, the pattern \(P\) is braided with braid index $\omega \geq 2$. By~\cite[Theorem~1.4]{Hayashi1999}, \(S\) gives rise to an essential \(n/\omega\)-string tangle decomposition of \(K\).
	By~\cite[Proposition~3.3]{HomSatellites}, see also \cite[Theorem~7.3]{BakerMotegi}, the companion knot of a satellite L-space knot is also an L-space knot.
	Since L-space knots are prime and also 2-string prime by the main result, $n/\omega \geq 3$.
	This means that $S$ intersects $P(K)$ at least
	$2 n \geq 12$ times.
\end{proof}

\subsection*{Outline} 
In Section~\ref{sec:review:lspace-knots}, we prove some technical lemmas about the knot Floer complexes of knots with L-space surgeries.  
In Section~\ref{sec:review:HFT}, we review the immersed curve theory for tangles.  
Section~\ref{sec:detection:split} is devoted to the proof of Theorem~\ref{thm:detection:split:Intro}. 
In Section~\ref{sec:pairings} we perform some calculations needed for the proof of the Main Theorem in Section~\ref{sec:proof-main-thm}.  

\subsection*{Acknowledgements} TL was partially supported by NSF grant DMS-1709702 and a Sloan fellowship. TL and CZ thank the hospitality of Virginia Commonwealth University where much of this research was performed. The authors would also like to thank Kimihiko Motegi for explaining Corollary~\ref{cor:satellites} to us, and Cameron Gordon, Jake Rasmussen, and Liam Watson for comments on an earlier draft of this paper. 
Finally, we would like to thank the anonymous referees for their careful review of this paper. 
We are particularly grateful to the referee who spotted a major mistake in the computation of what was previously Figure~10 and the proof of Proposition~5.1. 
Our subsequent correction simplified the overall argument.
\section{Knot Floer homology and properties of L-space knots}\label{sec:review:lspace-knots}
In the present section we will review the necessary background of knot Floer homology, and derive some structural properties of knots admitting a non-trivial surgery to an L-space. We assume the reader is familiar with knot Floer homology in some capacity; our main reference is \cite{HFK}, and an excellent survey can be found in \cite{Manolescu:intro}. All Floer homologies will be computed with coefficients in \(\F = \Z/2\Z\). 

With a knot \(K\subset S^3\) is associated a doubly-filtered chain complex \(\CFK^\infty(K)\) over the polynomial ring \(\F[U,U^{-1}]\), called the full knot Floer complex of \(K\). This complex is freely generated by the intersections of two Lagrangians in a symmetric product of a Riemann surface, constructed with analytical input from a Heegaard diagram. The generators of the complex carry two integer gradings called the \emph{Maslov} grading, $M$, and the \emph{Alexander} grading, $A$.  There is also a third grading, $\delta$, defined as $A - M$.  These are collectively referred to as the bigrading, as any two of these gradings determine the third.  The differential \(\partial^\infty\) of \(\CFK^\infty(K)\) decreases the Maslov grading by one and does not raise the Alexander grading. The action of the variable \(U\) decreases the Maslov grading by two and the Alexander grading by one. 

It is often convenient to visualize generators (over $\F$) of the full knot Floer complex \(C\coloneqq\CFK^\infty(K)\) as points in the \((i, j)\)-plane. Note that this picture does not take into account the Maslov or $\delta$-grading.  Here, \(i\) represents the (negative of the) \(U\)-exponent of a generator, and \(j\) is the Alexander grading.  Up to chain homotopy equivalence, we may assume that the differential $\partial^\infty$ strictly lowers one of the coordinates.    
For example, the complex corresponding to the torus knot \(T(3, 4)\) is pictured in Figure \ref{fig:CFKinfty-good}.  

Let $\partial$ denote the restriction of $\partial^\infty$ to $C\{i = 0\}$.  Then, $(C\{i=0\},\partial)$ is precisely $\CFhat(S^3)$ filtered by $K$.  Hence, the total homology with respect to $\partial$ recovers $\HFhat(S^3)$, and in particular, is one-dimensional.  The associated graded complex corresponding to the Alexander filtration on \(C\{i=0\}\) is \(\HFKhat(K)\) by our choice of model for $\CFK^\infty$.  This can be rephrased as saying there is a differential $\partial$ on \(\HFKhat(K)\) which decreases the Maslov grading by one, strictly decreases the Alexander grading, and has one-dimensional homology supported in Maslov grading zero.  Finally, there is an $(i,j)$-symmetry, in the sense that there is a chain homotopy equivalence from $C\{i = 0\}$ to $C\{j = 0\}$, which sends $C\{(0,a)\}$ to $C\{(a,0)\}$.   

Recall that an L-space is a rational homology sphere \(Y\) for which \(|H_1(Y;\Z)| = \dim \HFhat(Y)\).
Given a knot \(K\subset S^3\), one may ask for which slopes \(\tfrac{p}{q}\in\Q\), the three-manifold \(S^3_{p/q}(K)\) obtained by Dehn surgery of \(K\) along the slope \(\tfrac{p}{q}\) is an L-space. If there exists such a slope, we call \(K\) an L-space knot. 
If this slope can be chosen to be positive, we call \(K\) a positive L-space knot. Similarly, one defines negative L-space knots. (Since \(0\)-surgery along any knot in \(S^3\) yields a three-manifold which is not a rational homology sphere, any L-space knot is either positive or negative. Moreover, the unknot is the only L-space knot in \(S^3\) which is both positive and negative. This last fact follows from Ozsv\'ath and Szabo's Theorem~\ref{thm:lspace-ordered-stronger} below and unknot detection of knot Floer homology~\cite{OS:genus}.) 
There is a relationship between the knot Floer homology of \(K\) and the Heegaard Floer homology of sufficiently large integral surgery \(\HFhat(S^3_p(K))\) \cite{HFK}. From this relationship, Ozsváth and Szabó derived the main structural theorem for the knot Floer homology of an L-space knot \cite[Theorem 1.2]{HFKlens} (see also \cite[Theorem 2.10]{OSSConcordance}): 

\begin{theorem}\label{thm:lspace-ordered-stronger}
Suppose \(K \subset S^3 \) is a positive L-space knot. Then there exists a basis \(\{ x_{-\ell},\dots, x_\ell\}\) for \(\CFK^\infty(K)\) with the following properties: 
	\begin{enumerate}
		\item \(A(x_k) = A_k\), where \(A_{-\ell}<A_{-\ell+1} <\cdots < A_k <\cdots <A_{\ell-1}<A_{\ell}\), 
		\item \(A_k = -A_{-k}\), 
		\item If \(k\equiv \ell+1\) mod 2, then \(\partial^\infty (x_k)=x_{k-1} + U^{A_{k+1} - A_k}x_{k+1}\) and $M(x_{k}) - M(x_{k-1}) = 1$,
		\item\label{lspace:A} If \(k\equiv \ell\) mod 2, then \(\partial^\infty (x_k)=0\) and $M(x_k) - M(x_{k-1}) = 2(A_k - A_{k-1}) - 1$,
		\item \(x_\ell\) is the generator of \(H_*(C\{i=0\}, \partial) \cong \HFhat(S^3)\). 
	\end{enumerate}
\end{theorem}
While Theorem~\ref{thm:lspace-ordered-stronger} is technically only proved for knots with a positive \emph{integral} L-space surgery, the set of L-space slopes for a non-trivial positive L-space knot is exactly $[2g(K)-1,\infty)$ \cite[Proposition 9.6]{HFKQ}, so this is not an issue (see also \cite{KMOS}).  For negative L-space knots, a result similar to Theorem~\ref{thm:lspace-ordered-stronger} follows from the Alexander symmetry \cite{HFK}:
\[
\HFKhat_{-M}(mK,-A) \cong \HFKhat_M(K,A).
\]
The constraints of Theorem~\ref{thm:lspace-ordered-stronger} imply that \(\CFK^\infty\) of a knot admitting a non-trivial L-space surgery displays a ``staircase structure'', as is shown in Figure \ref{fig:CFKinfty-good}.  
Let us also observe that for a positive L-space knot, the Alexander grading of the top-most generator, the three-genus, smooth four-ball genus and \(\tau\) invariant agree \cite{OS:genus, HFKlens}, that is,
\[
A(x_\ell) = A_\ell = g(K) = \tau(K) = g_4(K).
\]

\begin{corollary}
\label{cor:lspace-ordered-short}
For any L-space knot \(K\),
\[
\HFKhat(K) \cong \bigoplus_{k = -\ell}^\ell \F_{(M_k, A_k)},
\]
where \(M_k < M_{k+1}\) and \(A_k < A_{k+1}\) for all \(k\).  
\qed 
\end{corollary}


\begin{corollary}
\label{cor:gaps}
For any L-space knot \(K\), $|\delta_k - \delta_{k-1}| = A_k - A_{k-1} - 1$.  
More generally, for any $k > k'$,
\[
| \delta_k - \delta_{k'}| \leq A_k - A_{k'} - (k-k').
\]
\end{corollary}

\begin{proof}
Recall that the \(\delta\)-grading is defined by the equation $\delta = A-M$. 
Suppose \(K\) is a positive L-space knot. 
If $k\equiv \ell+1$ mod 2, then by Theorem \ref{thm:lspace-ordered-stronger} (3), $M(x_{k}) - M(x_{k-1}) = 1$, in which case $\delta_k - \delta_{k-1} = A_k- A_{k-1} -1$. Otherwise, if $k\equiv \ell$ mod 2, then by (4), we have $M(x_k) - M(x_{k-1}) = 2(A_k - A_{k-1}) - 1$, and so $\delta_k - \delta_{k-1} = -(A_k- A_{k-1} -1)$. In either case, the first statement holds. 
If \(K\) is a negative L-space knot, the statement follows from the previous case and the Alexander symmetry mentioned above. 
The second statement follows from the first and the triangle inequality. 
\end{proof}

\begin{remark}
	\label{rem:gaps}
	Let us pause for a moment to interpret Corollaries~\ref{cor:lspace-ordered-short} and~\ref{cor:gaps} graphically. Recall that we may visualize the knot Floer complex \(\CFK^\infty(K)\) in terms of points in the \((i, j)\)-plane and that \(\HFKhat(K)\) is equal to the restriction of \(\CFK^\infty(K)\) to \(j=0\) (or equivalently to \(i=0\)). In other words, we can visualize the Alexander grading of \(\HFKhat(K)\) graphically on the real line \(\R\) by placing each generator \(\bullet\in \HFKhat(K)\) at the position \(A(\bullet)\in\Z\subset\R\). For example, we may represent the relative Alexander grading of \(\HFKhat(K)\) of the torus knot \(K=T(3,4)\) as follows:
	\[
	\begin{tikzpicture}[scale=0.5]
		\draw[->] (-4,0) -- (4,0) node[right]{\(A\)};
		\draw (-3,-0.25) -- (-3,0.25) node[above] {\(\bullet\)};
		\draw (-2,-0.25) -- (-2,0.25) node[above] {\(\bullet\)};
		\draw (-1,-0.25) -- (-1,0.25);
		\draw (0,-0.25) -- (0,0.25) node[above] {\(\bullet\)};
		\draw (1,-0.25) -- (1,0.25);
		\draw (2,-0.25) -- (2,0.25) node[above] {\(\bullet\)};
		\draw (3,-0.25) -- (3,0.25) node[above] {\(\bullet\)};
	\end{tikzpicture}
	\]
	If \(K\) is an L-space knot, then by Corollary~\ref{cor:lspace-ordered-short}, there is at most one generator in each Alexander grading, and the first statement of Corollary~\ref{cor:gaps} can be interpreted as saying that the difference in $\delta$-gradings \(|\delta_k-\delta_{k-1}|\) measures the length of the `gap' between the Alexander gradings of \(x_{k-1}\) and \(x_k\):
	\[
	\begin{tikzpicture}[scale=0.5]
	\draw[->] (-3,0) -- (3,0) node[right]{\(A\)};
	\draw (-2,-0.25) -- (-2,0.25) node[above] {\(\bullet\)} ++(0,0.5) node[above] {\(\scriptstyle x_{k-1}\)};
	\draw (-1,-0.25) -- (-1,0.25);
	\draw (0,0.25) node [above] {\(\cdots\)};
	\draw (1,-0.25) -- (1,0.25);
	\draw (2,-0.25) -- (2,0.25) node[above] {\(\bullet\)} ++(0,0.5) node[above] {\(\scriptstyle x_{k}\)};
	\draw [decorate,decoration={brace}] (1.25,-0.4) -- (-1.25,-0.4) node[midway, below] {\(\scriptstyle A_k-A_{k-1}-1\)};
	\end{tikzpicture}
	\]
\end{remark}

\begin{lemma}\label{lem:convex}
Suppose \(K\) is an L-space knot and \(W\subsetneq\HFKhat(K)\) is a subspace which is relatively bigraded isomorphic to the knot Floer homology of an L-space knot.  
Then for any three generators \(x_{k_-}\), \(x_{k_0}\), and \(x_{k_+}\) with \(k_- < k_0 < k_+\), that is
\[
\begin{tikzpicture}[scale=0.5]
\draw[->] (-4,0) -- (4,0) node[right]{\(A\)};
\draw (-3,-0.25) -- (-3,0.25) node[above] {\(\bullet\)} ++(0,0.5) node[above] {\(\scriptstyle x_{k_-}\)};
\draw (-1.5,0.25) node [above] {\(\cdots\)};
\draw (0,-0.25) -- (0,0.25) node[above] {\(\circ\)} ++(0,0.5) node[above] {\(\scriptstyle x_{k_0}\)};
\draw (1.5,0.25) node [above] {\(\cdots\)};
\draw (3,-0.25) -- (3,0.25) node[above] {\(\bullet\)} ++(0,0.5) node[above] {\(\scriptstyle x_{k_+}\)};
\end{tikzpicture}
\]
\(x_{k_-} \in W\) and \(x_{k_+} \in W\) implies \(x_{k_0} \in W\).   
\end{lemma}
\begin{proof}
Let $J$ be a knot with relatively bigraded knot Floer homology isomorphic to $W$.  Suppose \(x_{k_0} \not\in W\). Without loss of generality, we may assume that there is no $k_- < k_* < k_+$ such that $x_{k_*} \in W$. Then the second part of Corollary~\ref{cor:gaps} applied to \(K\) implies 
\[
|\delta_{k_-} - \delta_{k_+}| \leq A_{k_+} - A_{k-} -2.
\]  
However, by the first part of Corollary~\ref{cor:gaps} applied to $J$, we have 
\[
|\delta_{k_-} - \delta_{k_+}| = A_{k_+} - A_{k_-} - 1,
\]
which is a contradiction.   
\end{proof}

Finally, we recall an observation, attributed to Rasmussen \cite{HeddenWatson}, that for an L-space knot the first two Alexander gradings are consecutive:
\begin{proposition}
\label{top 2}
If \(K\) is an L-space knot, with its knot Floer homology described as in Corollary~\ref{cor:lspace-ordered-short}, then \(A(x_\ell) - A(x_{\ell-1})= A_\ell - A_{\ell-1} = 1\).
\end{proposition}

The converse to Corollary~\ref{cor:lspace-ordered-short} also holds:

\begin{lemma}\label{lem:lspacestructure}
Suppose that 
\[
\HFKhat(K) \cong \bigoplus_{k = -\ell}^{\ell} \F_{(M_k, A_k)},
\]
where \(M_k < M_{k+1}\) and \(A_k < A_{k+1}\) for all \(k\).  Then \(K\) is an L-space knot.   
\end{lemma}

This lemma is perhaps known to some experts, but we include a proof nonetheless because the lemma will be crucial to the arguments which follow. 
Our arguments rely on \(\CFK^\infty(K)\) and Ozsváth and Szabó's large surgery formula. There also exists a more geometric proof which interprets \(\HFKhat(K)\) in terms of Hanselman, Rasmussen, and Watson's immersed curve invariant \(\HFhat(S^3\smallsetminus\mathring{\nu}(K))\) for three-manifolds with torus boundary \cite{HRW}.

\begin{proof}[Proof of Lemma~\ref{lem:lspacestructure}]
For notation, let \(x_k\) denote the single generator with Maslov and Alexander grading \((M_k, A_k)\).  
Recall that the total differential \(\partial\) on \(\HFKhat(K)\) lowers Maslov grading by one, strictly decreases Alexander grading, and has one-dimensional homology.  
Since the strict orders by Alexander gradings and Maslov gradings agree, we see that there exists some index \(-\ell \leq k_*\leq \ell\) such that \(\partial(x_{k_*}) = 0\) and \(x_{k_*}\not\in \im(\partial)\); in other words, \(x_{k_\ast}\) generates the total homology \(\HFhat(S^3)\cong \F\).  
Further, since the ordering of generators is by both Maslov and Alexander gradings, aside from \(k_*\), we may match consecutive indices such that \(\partial(x_k) = x_{k-1}\).  
Note that \(A(x_{k_*})\) is precisely the invariant \(\tau(K)\), by definition \cite{OS:FourBallGenus}.  

We will reconstruct \(\CFK^\infty(K)\) from this data.  As before, let \(C = \CFK^\infty(K)\); the data we have specified computes \(C\{i = 0\}\). Consider now \(C\{j=0\}\) and note that it consists of \(U\)-translates of the \(x_k\).  
In particular, \(C\{j = 0\}\) has a basis given by \(U^{A_k} x_k\), which has coordinates $(-A_k,0)$ in the \((i,j)\)-plane.  
The \((i, j)\)-symmetry of \(\CFK^\infty\) in fact implies that the differential in \(C\{j = 0\}\) sends \(U^{A_{-k}} x_{-k}\) to \(U^{A_{-k+1}} x _{-k+1}\) precisely when \(\partial(x_k) = x_{k-1}\).  
This is illustrated in Figure~\ref{fig:CFKinfty-symmetry}.  Extending by \(U\)-equivariance gives the complete complex \(\CFK^\infty\), up to possible diagonal arrows, that is, possibly excluding terms which both pick up powers of \(U\) {\em and} strictly decrease the Alexander grading.  
We will write \(\partial^\infty_0\) for the component of the total differential without diagonal arrows.  

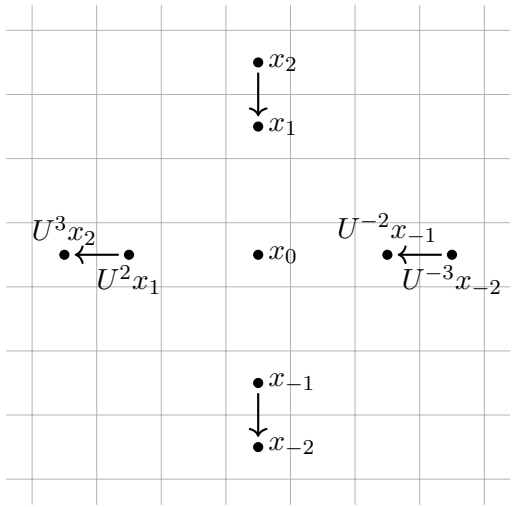
\begin{figure}[t]
\begin{center}
\begin{tikzpicture}[scale=\tikzscale]
	\draw[step=1, black!30!white, very thin] (-3.4, -3.4) grid (4.4, 4.4);
	
	\filldraw (0.5, 0.5) circle (2pt) node[right] {\(x_0\)};
	\node (A1) at (0.5, 2.5){};
	\node (A2) at (0.5, 3.5){};
	\filldraw (A1) circle (2pt) node [right] {\(x_1\)};
	\filldraw (A2) circle (2pt) node [right] {\(x_2\)};
	\node (B1) at (0.5, -1.5){};
	\node (B2) at (0.5, -2.5){};
	\filldraw (B1) circle (2pt) node [right] {\(x_{-1}\)};
	\filldraw (B2) circle (2pt) node [right] {\(x_{-2}\)};
	\node (C1) at (-2.5, 0.5){};
	\node (C2) at (-1.5, 0.5){};
	\filldraw (C1) circle (2pt) node [above] {\(U^3 x_2\)};
	\filldraw (C2) circle (2pt) node [below] {\(U^2 x_1\)};
	\node (D1) at (2.5, 0.5){};
	\node (D2) at (3.5, 0.5){};
	\filldraw (D1) circle (2pt) node [above] {\(U^{-2} x_{-1}\)};
	\filldraw (D2) circle (2pt) node [below] {\(U^{-3} x_{-2}\)};

	\draw [thick, <-] (A1) -- (A2);
	\draw [thick, ->] (B1) -- (B2);
	\draw [thick, <-] (C1) -- (C2);
	\draw [thick, <-] (D1) -- (D2);
\end{tikzpicture}
\caption{A superposition of the complexes \(C\{i = 0\}\) and \(C\{j= 0\}\) for a purported example illustrating the \((i,j)\)-symmetry when \(k_* = 0\) and \(\ell = 2\).  Note that the Maslov grading of \(x_0\) is 0, while the Maslov grading of $x_1$ determines the Maslov gradings of all generators other than $x_0$ by symmetry and the \(U\)-action.}
\label{fig:CFKinfty-symmetry}
\end{center}
\end{figure}
We consider three cases:  

\noindent {\bf Case 1:} \(\bm{k_* = \ell.}\quad\) 
In this case, we see that \((\CFK^\infty(K), \partial^\infty_0)\) has the ``staircase'' structure from Theorem~\ref{thm:lspace-ordered-stronger}.  By this, we mean that the complex is of the form 
\[
\partial^\infty_0(x_{k})  = \begin{cases} x_{k-1} + U^{A_{k+1} - A_{k}}x_{k+1}  \ & \text{ if } k \equiv \ell + 1\mod{2}, \\
0 & \text{ if } k \equiv \ell \mod{2}.
\end{cases}
\]
See Figure~\ref{fig:CFKinfty-good} for a specific example.

\begin{figure}[t]
	\centering
	\begin{subfigure}{0.45\textwidth}
		\centering
		\begin{tikzpicture}[scale=\tikzscale]
		\draw[step=1, black!30!white, very thin] (-3.4, -6.4) grid (2.4, 5.4);
		
		\filldraw (-2.5, .5) circle (2pt) node[] (A){};
		\filldraw (-2.5, -.5) circle (2pt) node[] (B){};
		\filldraw (-2.5, -2.5) circle (2pt) node[] (C){};
		\filldraw (-2.5, -4.5) circle (2pt) node[] (D){};	
		\filldraw (-2.5, -5.5) circle (2pt) node[] (E){};
		\filldraw (-1.5, 1.5) circle (2pt) node[] (K){};
		\filldraw (-1.5, 0.5) circle (2pt) node[] (L){};
		\filldraw (-1.5, -1.5) circle (2pt) node[] (M){};
		\filldraw (-1.5, -3.5) circle (2pt) node[] (N){};	
		\filldraw (-1.5, -4.5) circle (2pt) node[] (O){};
		\filldraw (-0.5, 2.5) circle (2pt) node[] (F){};
		\filldraw (-0.5, 1.5) circle (2pt) node[] (G){};
		\filldraw (-0.5, -.5) circle (2pt) node[] (H){};
		\filldraw (-0.5, -2.5) circle (2pt) node[] (I){};	
		\filldraw (-0.5, -3.5) circle (2pt) node[] (J){};
		\filldraw (0.5, 3.5) circle (2pt) node[] (P){};
		\filldraw (0.5, 2.5) circle (2pt) node[] (Q){};
		\filldraw (0.5, 0.5) circle (2pt) node[] (R){};
		\filldraw (0.5, -1.5) circle (2pt) node[] (S){};	
		\filldraw (0.5, -2.5) circle (2pt) node[] (T){};
		\filldraw (1.5, 4.5) circle (2pt) node[] (U){};
		\filldraw (1.5, 3.5) circle (2pt) node[] (V){};
		\filldraw (1.5, 1.5) circle (2pt) node[] (W){};
		\filldraw (1.5, -.5) circle (2pt) node[] (X){};	
		\filldraw (1.5, -1.5) circle (2pt) node[] (Y){};
	
		\draw [thick, ->] (B) -- (C);
		\draw [thick, ->] (D) -- (E);
		\draw [thick, ->] (G) -- (H);
		\draw [thick, ->] (I) -- (J);
		\draw [thick, ->] (L) -- (M);
		\draw [thick, ->] (N) -- (O);
		\draw [thick, ->] (Q) -- (R);
		\draw [thick, ->] (S) -- (T);
		\draw [thick, ->] (V) -- (W);
		\draw [thick, ->] (X) -- (Y);
	
		\draw [thick, ->] (G) -- (K);
		\draw [thick, ->] (V) -- (P);
		\draw [thick, ->] (Q) -- (F);
		\draw [thick, ->] (L) -- (A);
	
		\draw [thick, ->] (X) -- (H);
		\draw [thick, ->] (S) -- (M);
		\draw [thick, ->] (I) -- (C);
		\end{tikzpicture}
		\caption{}
		\label{fig:CFKinfty-good}
	\end{subfigure}
	\begin{subfigure}{0.45\textwidth}
		\centering
		\begin{tikzpicture}[scale=\tikzscale]
		\draw[step=1, black!30!white, very thin] (-3.4, -6.4) grid (2.4, 5.4);
		
		\filldraw (-2.5, .5) circle (2pt) node[] (A){};
		\filldraw (-2.5, -.5) circle (2pt) node[] (B){};
		\filldraw (-2.5, -2.5) circle (2pt) node[] (C){};
		\filldraw (-2.5, -4.5) circle (2pt) node[] (D){};	
		\filldraw (-2.5, -5.5) circle (2pt) node[] (E){};
		\filldraw (-1.5, 1.5) circle (2pt) node[] (K){};
		\filldraw (-1.5, 0.5) circle (2pt) node[] (L){};
		\filldraw (-1.5, -1.5) circle (2pt) node[] (M){};
		\filldraw (-1.5, -3.5) circle (2pt) node[] (N){};	
		\filldraw (-1.5, -4.5) circle (2pt) node[] (O){};
		\filldraw (-0.5, 2.5) circle (2pt) node[] (F){};
		\filldraw (-0.5, 1.5) circle (2pt) node[] (G){};
		\filldraw (-0.5, -.5) circle (2pt) node[] (H){};
		\filldraw (-0.5, -2.5) circle (2pt) node[] (I){};	
		\filldraw (-0.5, -3.5) circle (2pt) node[] (J){};
		\filldraw (0.5, 3.5) circle (2pt) node[] (P){};
		\filldraw (0.5, 2.5) circle (2pt) node[] (Q){};
		\filldraw (0.5, 0.5) circle (2pt) node[] (R){};
		\filldraw (0.5, -1.5) circle (2pt) node[] (S){};	
		\filldraw (0.5, -2.5) circle (2pt) node[] (T){};
		\filldraw (1.5, 4.5) circle (2pt) node[] (U){};
		\filldraw (1.5, 3.5) circle (2pt) node[] (V){};
		\filldraw (1.5, 1.5) circle (2pt) node[] (W){};
		\filldraw (1.5, -.5) circle (2pt) node[] (X){};	
		\filldraw (1.5, -1.5) circle (2pt) node[] (Y){};
		
		\draw [thick, ->] (A) -- (B);
		\draw [thick, ->] (D) -- (E);
		\draw [thick, ->] (F) -- (G);
		\draw [thick, ->] (I) -- (J);
		\draw [thick, ->] (K) -- (L);
		\draw [thick, ->] (N) -- (O);
		\draw [thick, ->] (P) -- (Q);
		\draw [thick, ->] (S) -- (T);
		\draw [thick, ->] (U) -- (V);
		\draw [thick, ->] (X) -- (Y);
		\draw [thick, ->] (G) -- (K);
		\draw [thick, ->] (V) -- (P);
		\draw [thick, ->] (Q) -- (F);
		\draw [thick, ->] (L) -- (A);
		\draw [thick, ->] (Y) -- (S);
		\draw [thick, ->] (T) -- (I);
		\draw [thick, ->] (J) -- (N);
		\draw [thick, ->] (O) -- (D);
		\end{tikzpicture}
		\caption{}
		\label{fig:CFKinfty-bad}
	\end{subfigure}
	\caption{Two examples of vertical and horizontal differentials in \(\CFK^\infty(K)\). The example (a) illustrates the case \(k_* = \ell\) in the proof of Lemma~\ref{lem:lspacestructure} for \(\ell = 2\). This is in fact \(\CFK^\infty(T_{3,4})\). The example (b) illustrates the case \(k_* = 0\) when \(\ell = 2\). As we will see, this complex cannot be realized by any knot.}
\end{figure}
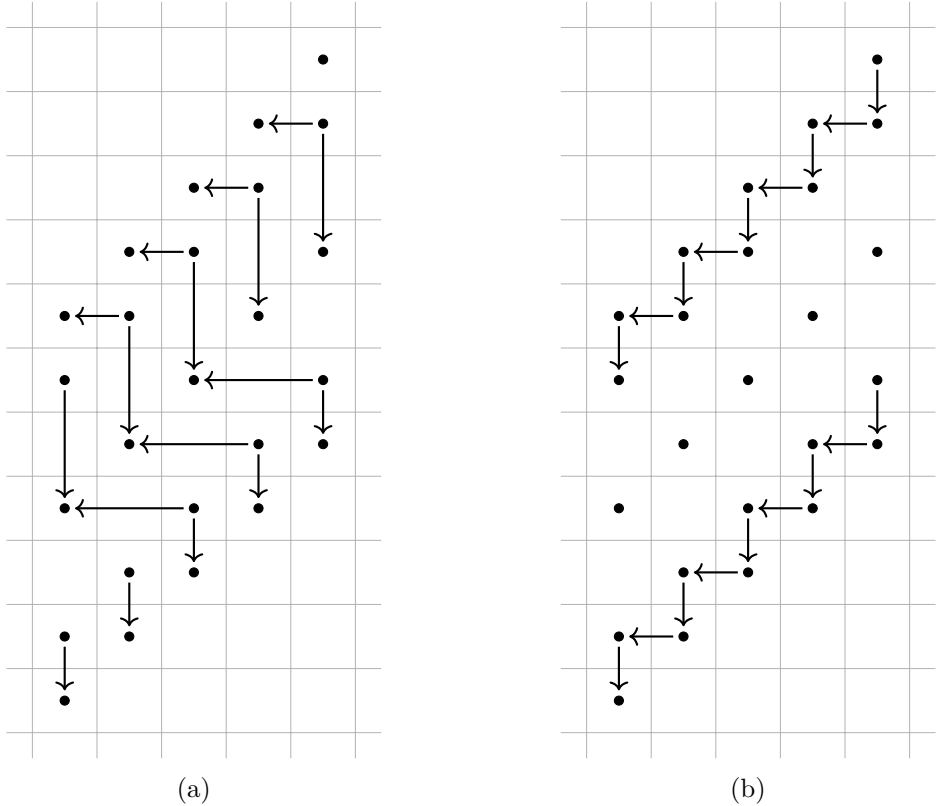

It follows from the large surgery formula of \cite{HFK} that \(K\) is a positive L-space knot if 
\[
H_*(C\{\max\{i,j - s\} = 0\})
\]
is one-dimensional for every \(s\).  It is easy to see that a knot Floer complex with a staircase structure satisfies this condition.  Further, the addition of diagonal arrows in \(\partial^\infty\) does not affect the complex \(C\{\max\{i,j-s\} = 0\}\), and so we see that \(K\) is an L-space knot, regardless of any diagonal arrows.  Indeed, any diagonal arrow decreases both \(i\) and \(j\) by at least 1, and hence its image is outside of the subquotient complex \(C\{\max\{i,j-s\} = 0\}\).

\bigskip
\noindent {\bf Case 2:} \(\bm{k_* = -\ell.}\quad\)
Since \(\HFKhat_M(mK,A) \cong \HFKhat_{-M}(K,-A)\), the mirror of \(K\) also satisfies the hypotheses of the lemma.  Further, the mirror has \(k_* = \ell\), 
since \(\tau\) changes sign under mirroring.  Applying the previous case, we see that the mirror of \(K\) is a positive L-space knot, so \(K\) is a negative L-space knot.

\bigskip
\noindent {\bf Case 3:} \(\bm{-\ell < k_* < \ell.}\quad\)
We will show by contradiction this case does not happen.  
Note that in fact we must have \(-\ell+1 < k_* < \ell-1\), for if \(k_* = \ell-1\), then \(\partial (x_\ell) = \partial(x_{\ell-1}) = 0\), which implies the total homology of $C\{i = 0\}$ is at least two-dimensional, which is a contradiction.  
A similar argument applies when \(k_* = -\ell+1\) by mirroring.  
Therefore, \(\partial(x_\ell) = x_{\ell-1}\) and by $U$-equivariance and the \((i,j)\)-symmetry, we observe that in \(C\{j = 0\}\), the differential sends \(U^{A_{\ell-1}}x_{\ell-1}\) to \(U^{A_\ell}x_\ell\).  
By \(U\)-equivariance and Proposition \ref{top 2} we have
\[
\partial^\infty_0(x_\ell) = x_{\ell-1} \qquad\text{ and }\qquad \partial^\infty_0(x_{\ell-1}) = U^{A_\ell - A_{\ell-1}} x_\ell = Ux_\ell.
\]
See for example Figure~\ref{fig:CFKinfty-bad}.  

We claim in fact that \(\partial^\infty_0(x_\ell) = \partial^\infty(x_\ell)\).  
Suppose \(U^c x_k\) appears in \(\partial^\infty(x_\ell)\).  
Since the \(U\)-powers provide a filtration on \(\CFK^\infty\), we know that \(c \geq 0\).  
Further, \(M(U^c x_k) = M(x_k) - 2c\).  
Therefore \(M(x_\ell) - M(U^c x_k) = 1\) if only if \(c = 0\) and \(k = \ell-1\), where the `if' direction follows from Proposition \ref{top 2}.  
Thus, \(\partial^\infty(x_\ell) = \partial^\infty_0(x_\ell)\).  
But now, \(\partial^\infty \circ \partial^\infty \) is non-zero on \(x_\ell\).  
This contradicts that \(\partial^\infty\) is a differential on \(\CFK^\infty(K)\).  
\end{proof}

We also make use of a more coarse property of L-space knots, for which we introduce notation.  

\begin{definition}
A bigraded vector space is \textbf{skeletal} if the rank of each summand of fixed Alexander grading is at most 1.  A knot is called skeletal if \(\HFKhat(K)\) is skeletal.
\end{definition}

\begin{example}\label{ex:LspaceKnot_are_skeletal}
	By Theorem~\ref{thm:lspace-ordered-stronger}, L-space knots are skeletal.
\end{example}


\begin{definition}\label{def:contiguous}
A bigraded vector space \(W\) is called \textbf{contiguous} if there exists a homogeneous basis \(\{x_1,\dots,x_k\}\) of \(W\) such that \(A(x_i)=A(x_{i+1})-1\) and \(\delta(x_i)\) is constant for all \(i=1,\dots, k\).
\end{definition}
	
Finally, we will collect the following property of two-bridge knots that will be useful later.  
\begin{lemma}\label{lem:2bridge}
Let \(Q_r, Q_s\) be rational tangles of slopes \(r, s\) respectively. 
If \(Q_r \cup Q_s\) is a skeletal knot then it is the torus knot \(T(2,n)\) for some odd \(n\).  In particular, the knot Floer homology is contiguous.   
\end{lemma}
\begin{proof}
It is well-known that two-bridge knots are alternating. By \cite[Proposition~4.1]{HFKlens}, the only alternating knots that are skeletal are the torus knots \(T(2,n)\) for odd \(n\).  The claim about gradings follows by combining Theorem~\ref{thm:lspace-ordered-stronger} with the structure of the Alexander polynomial of $T(2,n)$ torus knots.
\end{proof}

\section{Review of immersed curves for tangles}\label{sec:review:HFT}

In this section, we review some properties of the immersed curve invariant \(\HFT\) of four-ended tangles due to the third author \cite{PQMod, PQSym}.  We also establish some new results and tools that we will make use of in the proof of the Main Theorem.

\subsection{\texorpdfstring{The definition of \(\HFT\)}{The definition of HFT}}\label{sec:review:HFT:definition}

Given an oriented four-ended tangle \(T\) in a three-ball \(B^3\), the invariant \(\HFT(T)\) takes the form of a collection of immersed curves with local systems on the boundary of \(B^3\) minus the four tangle ends \(\partial T\).  
These immersed curves with local systems are defined in two steps, which we sketch below; for details, we refer the reader to~\cite{PQMod}. 

First, one fixes a particular auxiliary parametrization of \(\partial B^3\smallsetminus \partial T\) by four embedded arcs which connect the tangle ends. For example, the four gray dotted arcs in Figure~\ref{fig:HFT:example:tangle} define such a parametrization for the \((2,-3)\)-pretzel tangle. 
A tangle with such a parametrization can be encoded in a Heegaard diagram \((\Sigma,\bm{\alpha},\bm{\beta})\), where \(\Sigma\) is some surface with marked points. 
From this, one defines a relatively bigraded curved chain complex \(\CFTd(T)\) over a certain fixed \(\F\)-algebra \(\Ad\) as the multi-pointed Heegaard Floer theory of the triple \((\Sigma,\bm{\alpha},\bm{\beta})\), similar to Ozsváth and Szabó's link Floer homology~\cite{OSHFL}.  
One can show that the relatively bigraded chain homotopy type of \(\CFTd(T)\) is an invariant of the tangle \(T\) with the chosen parametrization \cite[Theorem~2.17]{PQMod}. 

The second step in the definition of \(\HFT(T)\) uses a classification of curved chain complexes over the algebra \(\Ad\). 
The classification says that the chain homotopy classes of bigraded chain complexes over \(\Ad\) are in one-to-one correspondence with free homotopy classes of bigraded immersed multicurves with local systems on a four-punctured sphere, which we will denote by~\(\FourPuncturedSphere\) \cite[Theorem~0.4]{PQMod}. 
By immersed curves, we mean immersions of \(S^1\) into \(\FourPuncturedSphere\) that define primitive elements of \(\pi_1(\FourPuncturedSphere)\); local systems are decorations of such immersed curves by invertible matrices over \(\F\), which are considered up to matrix similarity. 
Alternatively, one can regard local systems as vector bundles up to isomorphism, where \(\field\) is equipped with the discrete topology. 
Multiple parallel immersed curves (in the same bigrading) can be regarded as a single curve with a local system which is the direct sum of the individual local systems. 
This viewpoint, however, is only necessary for the classification result. 
Instead, we will often assume the opposite extreme viewpoint and treat the trivial \(n\)-dimensional local system as \(n\) distinct parallel curves. 
The bigrading on immersed curves, and in particular the Alexander grading, is described in more detail in Subsections~\ref{sec:review:HFT:Bigrading} and~\ref{sec:review:HFT:Alex} below. 
The correspondence between curved complexes and immersed multicurves uses a parametrization of \(\FourPuncturedSphere\) which is identical to the parametrization of \(\partial B^3\smallsetminus \partial T\). We will generally assume that the multicurves intersect the fixed parametrization minimally. Then, roughly speaking, these intersection points correspond to generators of the according curved chain complexes and paths between those intersection points correspond to the differentials. 
Finally, \(\HFT(T)\) is defined as the collection of relatively bigraded immersed curves on \(\FourPuncturedSphere\) which corresponds to the curved complex \(\CFTd(T)\). 
The orientation of the tangle \(T\) only affects the bigrading. 
One can show that the identification of \(\FourPuncturedSphere\) with \(\partial B^3\smallsetminus \partial T\) is natural; that is to say, if a tangle \(T'\) is obtained from \(T\) by adding twists to the tangle ends, the complex \(\CFTd(T')\) determines a new set of immersed curves \(\HFT(T')\), which agrees with the one obtained by twisting the immersed curves \(\HFT(T)\) accordingly \cite[Theorem~0.2]{PQSym}:

\begin{theorem}\label{thm:HFT:Twisting}
	The invariant \(\HFT\) commutes with the action of the mapping class group of \(\partial B^3\smallsetminus \partial T\). 
\end{theorem}

\begin{example}
  Figure~\ref{fig:HFT:example:Curve:Downstairs} shows the four-punctured sphere \(\FourPuncturedSphere\), drawn as the plane plus a point at infinity minus the four punctures labelled \(\TEI\), \(\TEII\), \(\TEIII\), and \(\TEIV\), together with the standard parametrization which identifies \(\FourPuncturedSphere\) with \(\partial B^3\smallsetminus \partial T\). 
  The dashed curve along with the two immersed curves winding around the punctures form the invariant \(\HFT(T_{2,-3})\) for the \((2,-3)\)-pretzel tangle \cite[Example~2.26]{PQMod}. 
  All three components of this invariant carry the (unique) one-dimensional local system. 
\end{example}

\begin{example}
  A (parametrized) tangle is called rational if it is obtained from the trivial tangle \(\InlineTrivialTangle\) by adding twists to the tangle ends. 
  The name rational tangle originated with Conway, who showed that these tangles are in one-to-one correspondence with fractions \(\tfrac{p}{q}\in\QPI\)~\cite{Conway}. 
  We denote the rational tangle corresponding to a slope \(s\in\QPI\) by \(Q_s\). 
  The invariant \(\HFT(Q_s)\) consists of a single embedded curve which is the boundary of a disk embedded into \(B^3\) that separates the two tangle strands of \(Q_s\) \cite[Example~2.25]{PQMod}. 
  The local system on this curve is one-dimensional.
\end{example}
It is known that \(\HFT\) detects rational tangles, as follows.

\begin{theorem}\cite[Theorem~6.2]{PQMod}
\label{thm:HFT:rational_tangle_detection}
A tangle \(T\) is rational if and only if \(\HFT(T)\) consists of a single embedded component carrying the unique one-dimensional local system. 
\end{theorem}

\begin{figure}[t]
  \centering
  \begin{subfigure}{0.28\textwidth}
  	\centering
  	\(\pretzeltangle\)
  	\caption{}\label{fig:HFT:example:tangle}
  \end{subfigure}
	\begin{subfigure}{0.28\textwidth}
		\centering
		\(\pretzeltangleDownstairs\)
		\caption{}\label{fig:HFT:example:Curve:Downstairs}
	\end{subfigure}
	\begin{subfigure}{0.4\textwidth}
		\centering
		\(\pretzeltangleUpstairs\)
		\caption{}\label{fig:HFT:example:Curve:Upstairs}
	\end{subfigure}
  \caption{(a) The pretzel tangle \(T_{2,-3}\) in the three-ball, (b) its invariant \(\HFT(T_{2,-3})\) in \(\FourPuncturedSphere\) consisting of a single rational component (the dashed curve) and a conjugate pair of special components (the solid curves), and (c) the lift of \(\HFT(T_{2,-3})\) to \(\PuncturedPlane\)}\label{fig:HFT:example}
\end{figure} 

\subsection{\texorpdfstring{A gluing theorem for \(\HFT\)}{A gluing theorem for HFT}}\label{sec:review:HFT:gluing}

The invariant \(\HFT(T)\) can be also defined via Zarev's bordered sutured Heegaard Floer theory~\cite{Zarev}. 
In this alternative construction, the curved chain complex \(\CFTd(T)\) is replaced by an (a posteriori equivalent) algebraic object, namely the bordered sutured type~D structure associated with the tangle complement which is equipped with a certain bordered sutured structure; for details, see~\cite[Section~5]{PQSym}. 
This perspective allows one to prove the following gluing result \cite[Theorem~5.9]{PQMod} which relates the invariant \(\HFT\) to link Floer homology \(\HFL\) via Lagrangian Floer homology \(\HF\).  For notation, throughout the rest of this article, let \(V\) be a two-dimensional vector space supported in a single relative \(\delta\)-grading and two consecutive relative Alexander gradings. Also, we will always assume that tangles are glued as in Figure~\ref{fig:tanglepairing}.

\begin{theorem}\label{thm:GlueingTheorem:HFT}
  Let \(L=T_1\cup T_2\) be the result of gluing two oriented four-ended tangles \(T_1\) and \(T_2\) together such that their orientations match. Let \(\mr(T_1)\) be the mirror image of \(T_1\) with the orientation of all components reversed. Then
  \[
  \HFL(L)\otimes V
  \cong
  \HF\left(\HFT(\mr(T_1)),\HFT(T_2)\right)
  \]
  if the four open components of the tangles become identified to the same component and 
  \[
  \HFL(L)
  \cong
  \HF\left(\HFT(\mr(T_1)),\HFT(T_2)\right)
  \]
  otherwise. 
\end{theorem}

The Lagrangian Floer homology \(\HF(\gamma,\gamma')\) of two immersed curves with local systems \(\gamma\) and \(\gamma'\) is a vector space generated by intersection points between the two curves. More precisely, one first arranges that the components are transverse and do not cobound immersed annuli; then, \(\HF(\gamma,\gamma')\) is the homology of the following chain complex: For each intersection point between \(\gamma\) and \(\gamma'\), there are \(n\cdot n'\) corresponding generators of the underlying chain module, where \(n\) and \(n'\) are the dimensions of the local systems of \(\gamma\) and \(\gamma'\), respectively. The differential is defined by counting certain bigons between these intersection points. As a consequence, the dimension of \(\HF(\gamma,\gamma')\) is equal to the minimal intersection number between the two curves times the dimensions of their local systems, provided that the curves are not parallel. If the curves are parallel, the dimension of \(\HF(\gamma,\gamma')\) may be greater than the minimal geometric intersection number for certain choices of local systems; for details, see~\cite[Sections~4.5 and~4.6, in particular Theorem~4.45]{PQMod}. For a more explicit example, suppose $\gamma$ and $\gamma'$ are parallel embedded curves of the same slope equipped with local systems of dimensions $n$ and $n'$ respectively.  Then, $\dim \HF(\gamma,\gamma')$ can realize any even number between 0 and $2 n \cdot n'$, depending on the local systems, even though the minimal geometric intersection number between these curves is zero.  Throughout, we will always assume that \(\gamma\) and \(\gamma'\) intersect minimally without cobounding immersed annuli.   

\(\HFT(\mr(T_1))\) can be easily computed from  \(\HFT(T_1)\). For this, let \(\mr\) be the involution of \(\FourPuncturedSphere\) whose fixed point set is the punctures and arcs in the parametrization and which interchanges the two components of the complement of the fixed set. Then one can show \cite[Definition~5.3 and Proposition~5.4]{PQMod}:

\begin{lemma}
  For any four-ended tangle \(T\), \(\HFT(\mr(T))\cong \mr(\HFT(T))\).
\end{lemma}

For example, \(\HFT(Q_{-s}) \cong \mr(\HFT(Q_s))\). 

In this paper, we are primarily interested in the statement of Theorem~\ref{thm:GlueingTheorem:HFT} for the case when \(T_1\cup T_2\) is a knot. We therefore introduce the following notation:
Suppose for two collections of bigraded immersed curves \(\Gamma_1\) and \(\Gamma_2\), there exists a bigraded vector space \(W\) which satisfies
\[
W\otimes V\cong \HF(\Gamma_1,\Gamma_2).
\]
Then we denote this vector space $W$ by \(\HFr(\Gamma_1,\Gamma_2)\). Note that it is well-defined up to relatively bigraded isomorphism, since \(\HF(\Gamma_1,\Gamma_2)\) is finite dimensional.  Hence, we have: 

\begin{theorem}\label{thm:GlueingTheorem:HFT:HFr}
  If \(T_1\cup T_2\) in Theorem~\ref{thm:GlueingTheorem:HFT} is a knot and \(\Gamma_i=\HFT(T_i)\) for \(i=1,2\), then 
  \begin{equation*}
    \HFKhat(T_1 \cup T_2)
    \cong
    \HFr\left(\mr(\Gamma_1),\Gamma_2\right).\qedhere
  \end{equation*}
\end{theorem}

\subsection{\texorpdfstring{The geography problem for components of \(\HFT\)}{The geography problem for components of HFT}}\label{sec:review:HFT:geography}

Often, it is useful to consider the immersed curves in a covering space of \(\FourPuncturedSphere\), namely the plane \(\mathbb{R}^2\) minus the integer lattice \(\mathbb{Z}^2\). 
We may think of \(\mathbb{R}^2\) as the universal cover of the torus which is the double branched cover of the sphere \(S^2\) branched at four marked points; then the integer lattice \(\mathbb{Z}^2\) is the preimage of the branch set.
This covering space is illustrated in Figure~\ref{fig:HFT:example:Curve:Upstairs}, where the standard parametrization of \(\FourPuncturedSphere\) has been lifted to \(\PuncturedPlane\) and the front face and its preimage under the covering map are shaded grey. 
The preimages of a puncture \(\TEi\) of \(\FourPuncturedSphere\) are labelled \(\tilde{\TEi}\).
This picture also includes the lifts of the curves in \(\HFT(T_{2,-3})\). 
The lift of the embedded (dashed) curve is a straight line of slope \(\tfrac{1}{2}\). 
More generally, the lift of an embedded curve which looks like the invariant of a \(\tfrac{p}{q}\)-rational tangle is a straight line of slope \(\tfrac{p}{q}\). We call such curves \textbf{rational} and denote them by \(\Rational(\tfrac{p}{q})\). 
The lifts of the two non-embedded components of  \(\HFT(T_{2,-3})\) look more complicated. Remarkably, however, this extremely simple example shows almost all the complexity of the immersed curves that can appear as components of \(\HFT(T)\) for four-ended tangles \(T\). 

To understand the geography of components of \(\HFT\) in general, observe that the linear action on the covering space \(\PuncturedPlane\) by \(\SL_2(\mathbb{Z})\) corresponds to Dehn twisting in \(\FourPuncturedSphere\), or equivalently, adding twists to the tangle ends \cite[Observation~3.2]{PQSym}.
We call a curve in \(\FourPuncturedSphere\) \textbf{special} if, after some twisting, it is equal to the curve \(\Special_n(0;\TEi,\TEj)\) whose lift to \(\PuncturedPlane\) is shown in Figure~\ref{fig:intro:curves}. Note that the lift of any such curve can be isotoped into an arbitrarily small neighborhood of a straight line of some rational slope \(\frac{p}{q}\in\QPI\) going through some punctures \(\tilde{\TEi}\) and \(\tilde{\TEj}\), in which case we denote this curve by \(\Special_n(\frac{p}{q};\TEi,\TEj)=\Special_n(\frac{p}{q};\TEj,\TEi)\). One can then show \cite[Theorem~0.5]{PQSym}:

\begin{figure}[t]
  \centering
  \(\rigidcurveIntro\)
  \caption{The lift of the curve \(\Special_n(0;\TEi,\TEj)\), where \(n\in\mathbb{N}\) and \((\TEi,\TEj)=(\TEIV,\TEI)\) or \((\TEII,\TEIII)\)}\label{fig:intro:curves}
\end{figure}

\begin{theorem}\label{thm:geography_of_HFT}
  For a four-ended tangle \(T\), the underlying curve of each component of \(\HFT(T)\) is either rational or special. 
  Moreover, if it is special, its local system is equal to an identity matrix. 
\end{theorem}

For example, we can now write
\(\HFT(T_{2,-3})\) as the union of the rational curve \(\Rat(\frac{1}{2})\) and the two special components \(\Special_1(0;\TEIV,\TEI)\) and \(\Special_1(0;\TEII,\TEIII)\). 
Special components for $n>1$ show up in the invariants of two-stranded pretzel tangles with more twists \cite[Theorem~6.9]{PQMod}. 
Whether rational components with non-trivial local systems appear in \(\HFT\) is currently not known. 

\subsection{\texorpdfstring{Properties of \(\HFT\)}{Properties of HFT}}\label{sec:review:HFT:Properties}

The following result is a simplified version of 
\cite[Theorem~0.10]{PQSym} which is sufficient for our purposes.

\begin{theorem}[Conjugation symmetry]\label{thm:Conjugation}
  Let \(\TEi\), \(\TEj\), \(\TEk\), and \(\TEl\) be integers such that \(\{\TEi,\TEj,\TEk,\TEl\}=\{\TEI,\TEII,\TEIII,\TEIV\}\). 
  Moreover, let \(\frac{p}{q}\in\QPI\). 
  Then, for any four-ended tangle \(T\), the numbers of components of the form \(\Special_n(\frac{p}{q};\TEi,\TEj)\) and \(\Special_n(\frac{p}{q};\TEk,\TEl)\) in \(\HFT(T)\) agree. 
\end{theorem}

There are also restrictions on rational components. The following is \cite[Observation 6.1]{PQMod}:

\begin{lemma}\label{lem:HFT_detects_connectivity}
  Each rational component of \(\HFT(T)\) separates the four punctures into two pairs, which agrees with how the two open components of \(T\) connect the tangle ends.
\end{lemma}

\begin{corollary}\label{cor:pairing-knots}
  If \(T_1 \cup T_2\) is a knot, then \(Q_r \cup Q_s\) is a knot for any rational slopes \(r\) and \(s\) appearing as rational curve components in \(\HFT(T_1)\) and \(\HFT(T_2)\). \qed
\end{corollary}

\begin{corollary}\label{cor:pairing-knots:local_systems}
  For \(i=1,2\), let \(T_i\) be a four-ended tangle and let \(\Gamma_i\) be the multicurve obtained from \(\HFT(T_i)\) by 
  replacing every non-trivial local system by a trivial one of the same dimension.
  If \(T_1 \cup T_2\) is a knot then \(\HFKhat(T_1 \cup T_2)\cong \HFr(\mr(\Gamma_1),\Gamma_2)\). 
\end{corollary}

\begin{proof}
  By Theorem~\ref{thm:geography_of_HFT}, only rational components may carry non-trivial local systems. By Corollary~\ref{cor:pairing-knots}, no rational component of \(\mr\HFT(T_1)\) is parallel to any rational component of \(\HFT(T_2)\). Theorem~4.45 from \cite{PQMod} implies that the Lagrangian Floer homology between two non-parallel curves is invariant under changing local systems, provided the dimensions of the local systems stay the same. We now conclude with Theorem~\ref{thm:GlueingTheorem:HFT:HFr}. 
\end{proof}

The following result is new:

\begin{proposition}\label{prop:Odd_number_of_rationals}
  For any tangle \(T\) without closed components, the number of rational curves in \(\HFT(T)\) weighted by the dimensions of their local systems is odd. So in particular, for such tangles, there is at least one rational curve whose local system has an odd dimension. 
\end{proposition}

\begin{proof}
	The tangle \(T\) admits a rational closure to a knot \(K=Q_{-s}\cup T\), and the rational curve \(\Rational(s)\) corresponding to this closure is not parallel to any component of \(\HFT(T)\) by Lemma~\ref{lem:HFT_detects_connectivity}. 
  The pairing of \(\Rational(s)\) with any closed immersed curve with local system is even-dimensional. 
  By conjugation symmetry (Theorem~\ref{thm:Conjugation}), the special components of \(\HFT(T)\) appear in pairs. Moreover, \(\Rational(s)\) intersects the two curves in each such pair in the same number of points, since the involution exchanging the two special curves preserves \(\Rational(s)\).  Consequently, the number of generators that the special components of \(\HFT(T)\) contribute to \(V\otimes \HFKhat(K)\) is divisible by 4. 
  Each rational component of \(\HFT(T)\) with local system \(X_i\) of dimension \(d_i\) contributes \(\mathbb{F}^{d_i}\otimes V\otimes \HFKhat(K_i)\) for some two-bridge knot \(K_i\), possibly shifted in bigrading, by Corollaries~\ref{cor:pairing-knots} and ~\ref{cor:pairing-knots:local_systems}.   
  Now, \(\HFKhat(K_i)\) and \(\HFKhat(K)\) have odd rank, so the claim follows.
\end{proof}

\subsection{\texorpdfstring{The bigrading on \(\HFT\)}{The bigrading on HFT}}\label{sec:review:HFT:Bigrading}

Like link Floer homology, the invariant \(\HFT\) comes with a relative bigrading. 
The following definition of an absolute bigrading on multicurves is a reformulation of \cite[Definitions~4.28 and~5.1]{PQMod}; see also \cite[Definition~1.7]{PQSym}.

Let us consider the \(\delta\)-grading first. 
The \(\delta\)-grading of an immersed multicurve \(\Gamma\) is a function 
\[
\delta\co \Gen(\Gamma)\longrightarrow \tfrac{1}{2}\mathbb{Z}
\]
where \(\Gen(\Gamma)\) is the set of intersection points between the parametrization of \(\FourPuncturedSphere\) and~\(\Gamma\), assuming as usual that this intersection is minimal. The function \(\delta\) is subject to the following compatibility condition: 
Suppose \(x,x'\in \Gen(\Gamma)\) are two intersection points such that there is a path \(\psi\)  on \(\Gamma\) which connects \(x\) to \(x'\) without meeting any parametrizing arc, except at the endpoints. We distinguish three cases, which are illustrated in Figure~\ref{fig:domains:xx}. The path can turn left (a), it can go straight across (b), or it can turn right (c). Then 
\[
\delta(x')-\delta(x)=
\begin{cases*}
\tfrac{1}{2} & if the path \(\psi\) turns left,\\
0 & if the path \(\psi\) goes straight across,\\
-\tfrac{1}{2} & if the path \(\psi\) turns right.\\
\end{cases*}
\]
The \(\delta\)-grading on \(\HFT(T)\) is well-defined up to a constant.

\begin{figure}[tb]
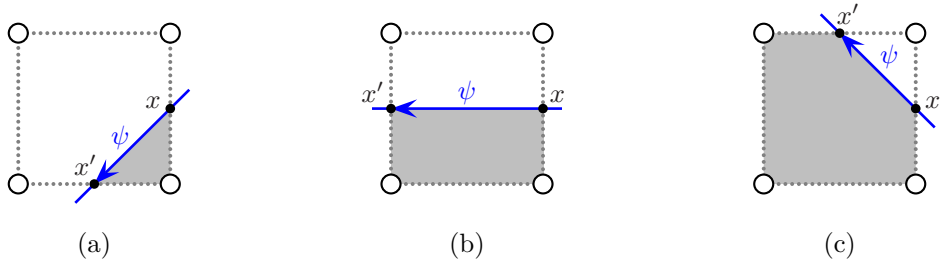

	\begin{subfigure}{0.3\textwidth}
		\centering
		\(\xI\)
		\caption{}\label{fig:domains:xx:1}
	\end{subfigure}
	\begin{subfigure}{0.3\textwidth}
		\centering
		\(\xII\)
		\caption{}\label{fig:domains:xx:2}
	\end{subfigure}
	\begin{subfigure}{0.3\textwidth}
		\centering
		\(\xIII\)
		\caption{}\label{fig:domains:xx:3}
	\end{subfigure}
	\caption{Basic regions illustrating the definition of the bigrading on a single curve}\label{fig:domains:xx}
\end{figure}

The Alexander grading is defined similarly. 
First, we define an \textbf{ordered matching} as a partition \(\{(\TEi_1,\TEo_1),(\TEi_2,\TEo_2)\}\) of \(\{\TEI,\TEII,\TEIII,\TEIV\}\) into two ordered pairs. A four-ended tangle gives rise to such an ordered matching, according to which pairs of tangle ends are connected; the order is determined by the orientation of the two open components of the tangle: We order each pair of tangle ends such that the inward pointing end comes first, the outward pointing end second. 
Given such an ordered matching, the Alexander grading is a function
\[
\Alex\co
\Gen(\Gamma)
\longrightarrow 
\GeneralAlex
\coloneqq
\mathbb{Z}^4/(e_{\TEi_1}+e_{\TEo_1},e_{\TEi_2}+e_{\TEo_2}),
\]
where \(e_{\TEj}\) is the \(\TEj^\text{th}\) unit vector in \(\mathbb{Z}^4\),
satisfying the following compatibility condition: 
If \(\psi\co x\rightarrow x'\) is a path as in the discussion of the \(\delta\)-grading above, then 
\[
\Alex(x')-\Alex(x)=(a_{\TEI},a_{\TEII},a_{\TEIII},a_{\TEIV})\in\GeneralAlex,
\text{ where }a_{\TEj}=
\begin{cases*}
-1 & if \(\TEj\) lies to the left of \(\psi\)\\
0 & if \(\TEj\) lies to the right of \(\psi\).
\end{cases*}
\]
The ordering on the pairs in our ordered matching then determines a homomorphism
\[
\GeneralAlex\longrightarrow\tfrac{1}{2}\mathbb{Z},\quad (a_{\TEI},a_{\TEII},a_{\TEIII},a_{\TEIV})\mapsto
\tfrac{1}{2}(\varepsilon_{\TEI}a_{\TEI}+
\varepsilon_{\TEII}a_{\TEII}+
\varepsilon_{\TEIII}a_{\TEIII}+
\varepsilon_{\TEIV}a_{\TEIV}),
\]
where \(\varepsilon_{\TEi_1}=\varepsilon_{\TEi_2}=+1\) and \(\varepsilon_{\TEo_1}=\varepsilon_{\TEo_2}=-1\).
This specifies a univariate Alexander grading 
\[
A\co \Gen(\Gamma)\longrightarrow \tfrac{1}{2}\mathbb{Z}.
\]
Given a four-ended tangle \(T\), the univariate Alexander grading \(A\) on \(\HFT(T)\) is defined using the ordered matching determined by \(T\). 
The grading is well-defined up to a constant. 
It corresponds to the univariate Alexander grading on link Floer homology. 
However, for general computations, it is usually more convenient to use \(\Alex\) instead as it does not require us to choose a particular orientation in advance. 
(One can also define a multivariate Alexander grading, but we do not need it in this paper.)

The third grading, the homological grading, is defined by $A-\delta$, and corresponds to the Maslov grading. When comparing our notation to~\cite{HDsForTangles,PQMod,PQSym}, note that our Alexander grading \(\Alex\) is denoted by \(A\) in those papers, while our \(A\) is equal to \(\tfrac{1}{2}\overline{A}\).

\subsection{\texorpdfstring{The Alexander grading of curves in the covering space for \(\HFT\)}{The Alexander grading of curves in the covering space for HFT}}\label{sec:review:HFT:Alex}

The purpose of this subsection is to develop tools that enable us to better understand the Alexander grading in terms of the covering space \(\PuncturedPlane\) of the four-punctured sphere \(\FourPuncturedSphere\). 
The ideas are very similar to the ones used in~\cite{KWZthinness} to study the \(\delta\)-grading. 
Most results in this subsection are true for arbitrary curves in \(\PuncturedPlane\). Nonetheless, we will implicitly assume throughout that all curves are lifts of rational or special curves.
For notation, any symbol decorated with a tilde \(\tilde{}\) will denote the lift to \(\PuncturedPlane\) of an object in \(\FourPuncturedSphere\) represented by the same symbol.
In the following, we will treat all points in the integer lattice as marked points. 
Let us denote by \(\ParaCovering\) the union of the integer lattice points with the preimage of the parametrization of \(\FourPuncturedSphere\).

\begin{definition}
	Suppose \(\Gamma=\{\tilde{\gamma}_1,\dots,\tilde{\gamma}_n\}\) is a set of curves in \(\PuncturedPlane\) which avoid the integer lattice points such that \(\ParaCovering\cup\Gamma=\ParaCovering\cup\tilde{\gamma}_1\cup\dots\cup\tilde{\gamma}_n\) is a planar graph whose vertices have all valence four. 
	This planar graph divides the plane into polygons, which we call \textbf{regions}. A \textbf{domain} is a formal linear combination of regions. In other words, a domain is an element of \(H_2(\mathbb{R}^2,\ParaCovering\cup\Gamma)\). 
	Then, given an integer lattice point \(\tilde{\circ}\) labelled by \(\tilde{\TEi}\) and a domain \(\varphi\), define an element \(\Alex(\varphi,\tilde{\circ})\in\GeneralAlex\) by setting its \(\TEi^{\text{th}}\) component equal to the sum of the multiplicities of the domain \(\varphi\) in the four regions adjacent to \(\tilde{\circ}\) and setting all other components equal to 0.
	We then define the Alexander grading \(\Alex(\varphi)\in\GeneralAlex\) of \(\varphi\) as the sum of \(\Alex(\varphi,\tilde{\circ})\), where \(\tilde{\circ}\) ranges over all integer lattice points.
  Note that \(\Alex(\varphi,\tilde{\circ})=0\) for all but finitely many integer lattice points \(\tilde{\circ}\), so \(\Alex(\varphi)\) is well-defined. 
  Also, the Alexander grading of domains is additive in the sense that for any two domains~\(\varphi\) and~\(\psi\), \(\Alex(\varphi+\psi)=\Alex(\varphi)+\Alex(\psi)\).
\end{definition}

\begin{remark}\label{rem:conventions}
	The figures in this paper follow the same conventions as in~\cite{PQMod}: We use the right-hand rule to determine the orientation of domains, and the normal vector fields of \(\FourPuncturedSphere\) and \(\PuncturedPlane\) are pointing into the page. 
	Thus, the boundary of a region of multiplicity \(+1\) is oriented clockwise. 
	For bigons such as the one in Figure~\ref{fig:lem:Alex:bigon}, this convention means that red Lagrangians are ``on the left'' and blue Lagrangians are ``on the right''.
\end{remark}

\begin{definition}\label{def:connecting_domain:same_curve}
  Given an absolutely Alexander graded curve \(\gamma\) (in the sense of Subsection~\ref{sec:review:HFT:Bigrading}), consider two intersection points \(\tilde{x}\) and \(\tilde{x}'\) of a lift \(\tilde{\gamma}\) with the integer lattice graph \(\ParaCovering\). A \textbf{connecting domain} from \(\tilde{x}\) to \(\tilde{x}'\) is a domain \(\varphi\in H_2(\mathbb{R}^2,\ParaCovering\cup\tilde{\gamma})\) with the property 
  \[
  \partial\Big(\partial\varphi\cap \tilde{\gamma}\Big)=\tilde{x}-\tilde{x}'.
  \] 
  For readers familiar with Heegaard Floer homology, it might be helpful to think of the curve \(\tilde{\gamma}\) as playing the role of a \(\beta\)-curve and \(\ParaCovering\) playing the role of an \(\alpha\)-curve.
\end{definition}

\begin{lemma}\label{lem:Alex:HFT:Euler:one_curve}
  With notation as in Definition~\ref{def:connecting_domain:same_curve}, 
  \[
  \Alex(x')-\Alex(x)=-\Alex(\varphi). 
  \]
\end{lemma}

\begin{proof}
	Figure~\ref{fig:domains:xx} shows the images of domains consisting of just a single region under the covering map \(\PuncturedPlane\rightarrow\FourPuncturedSphere\). For these domains, the lemma follows directly from the definition of the Alexander grading. So let us consider a general connecting domain \(\varphi\) from \(\tilde{x}\) to \(\tilde{x}'\). 
  By hypothesis, \(\partial\varphi\cap\tilde{\gamma}\) is a 1-chain connecting \(\tilde{x}'\) to \(\tilde{x}\). 
  Since \(\gamma\) is assumed to be rational or special, there are no cycles in this 1-chain, so it is simply a path from \(\tilde{x}'\) to \(\tilde{x}\). 
  Then, this path can be written as the intersection of \(\tilde{\gamma}\) with the boundary of \(\varphi'\), a connecting domain which is a sum of finitely many of the basic regions from Figure~\ref{fig:domains:xx}.
  The difference \(\varphi-\varphi'\) is a domain whose boundary lies entirely in \(\ParaCovering\), so it consists entirely of square regions, whose Alexander gradings vanish. 
\end{proof}

\begin{definition}\label{def:Alex:HFT:Euler:two_curves}
  Let \(\bullet\in\HF(\gamma,\gamma')\) be an intersection point between two  absolutely Alexander graded curves \(\gamma\) and \(\gamma'\). 
  Consider two lifts \(\tilde{\gamma}\) and \(\tilde{\gamma}'\) of these curves such that they intersect at a lift \(\tilde{\bullet}\) of the intersection point~\(\bullet\). A \textbf{connecting domain} for \(\tilde{\bullet}\) from \(\tilde{\gamma}\) to \(\tilde{\gamma}'\) is a domain \(\varphi\in H_2(\mathbb{R}^2,\ParaCovering\cup\tilde{\gamma}\cup \tilde{\gamma}')\) with the property 
  \[
  \partial\Big(\partial\varphi\cap \tilde{\gamma}\Big)=\tilde{x}-\tilde{\bullet}
  \quad
  \text{and}
  \quad
  \partial\Big(\partial\varphi\cap \tilde{\gamma}'\Big)=\tilde{\bullet}-\tilde{y}
  \quad
  \text{for some \(\tilde{x}\in\tilde{\gamma}\cap \ParaCovering\) and \(\tilde{y}\in\tilde{\gamma}'\cap \ParaCovering\).}
  \] 
\end{definition}

Given two bigraded curves \(\gamma\) and \(\gamma'\), one can define a bigrading on \(\HF(\gamma,\gamma')\). 
The reader unfamiliar with~\cite{PQMod} may take the following result as a definition for the Alexander grading on \(\HF(\gamma,\gamma')\).

\begin{lemma}\label{lem:Alex:HFT:Euler:two_curves}
	With notation as in Definition~\ref{def:Alex:HFT:Euler:two_curves}, 
  the Alexander grading of \(\bullet\) is equal to 
  \[
  \Alex(y)-\Alex(x)+\Alex(\varphi).
  \]
\end{lemma}

\begin{figure}[t]
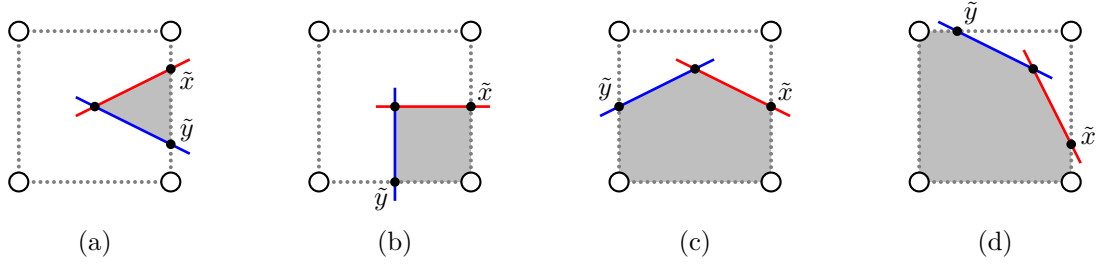

  \centering
  \begin{subfigure}{0.24\textwidth}
    \centering
    \(\xyI\)
    \caption{}\label{fig:domains:xy:1}
  \end{subfigure}
  \begin{subfigure}{0.24\textwidth}
    \centering
    \(\xyII\)
    \caption{}\label{fig:domains:xy:2}
  \end{subfigure}
  \begin{subfigure}{0.24\textwidth}
    \centering
    \(\xyIII\)
    \caption{}\label{fig:domains:xy:3}
  \end{subfigure}
  \begin{subfigure}{0.24\textwidth}
    \centering
    \(\xyIV\)
    \caption{}\label{fig:domains:xy:4}
  \end{subfigure}
  \caption{Basic connecting domains illustrating the first part of the proof of Lemma~\ref{lem:Alex:HFT:Euler:two_curves}}\label{fig:domains:xy}
\end{figure}

\begin{observation}\label{obs:Alex:reverse_order}
  The domain \(-\varphi\) is a connecting domain for a lift of the intersection point \(\bullet\) regarded as a generator of \(\HF(\gamma',\gamma)\). Therefore:
  \[
  \Alex(\bullet\in\HF(\gamma,\gamma'))
  =
  -\Alex(\bullet\in\HF(\gamma',\gamma)).
  \]
\end{observation}

\begin{proof}[Proof of Lemma~\ref{lem:Alex:HFT:Euler:two_curves}]
  If the domain \(\varphi\) consists of a single region of multiplicity 1, then up to rotation, there are only four cases, namely those shown in Figure~\ref{fig:domains:xy}. The lemma then follows directly from \cite[Definition~5.1]{PQMod}, since in each of those cases, the intersection point corresponds to some algebra element \(a\in\Ad\) and its Alexander grading \(\Alex(a)\) is equal to \(\Alex(\varphi)\) \cite[Definition~1.7]{PQSym}. 
  Let us now consider a general connecting domain \(\varphi\). Then near \(\tilde{\bullet}\), \(\varphi\) looks like one of the basic connecting domains \(\varphi'\) that we have just considered (up to adding multiples of square regions or basic regions from Figure~\ref{fig:domains:xx}). Suppose \(\varphi'\) connects \(\tilde{x}'\in\tilde{\gamma}\cap \ParaCovering\) to \(\tilde{y}'\in\tilde{\gamma}'\cap \ParaCovering\). Then, as we have just verified,
  \[
  \Alex(\bullet\in\HF(\gamma,\gamma'))
  =
  \Alex(y')-\Alex(x')+\Alex(\varphi').
  \]
  Let \(\varphi_x\) and \(\varphi_y\) be connecting domains from \(\tilde{x}\) to \(\tilde{x}'\) and from \(\tilde{y}'\) to \(\tilde{y}\), respectively. Then, by Lemma~\ref{lem:Alex:HFT:Euler:one_curve}, 
  \[
  \Alex(x')-\Alex(x)=-\Alex(\varphi_x)
  \quad\text{and}\quad
  \Alex(y)-\Alex(y')=-\Alex(\varphi_y).
  \]
  Combining all three relations, we see that 
  \[
  \Alex(\bullet\in\HF(\gamma,\gamma'))
  =
  \Alex(y)-\Alex(x)+\Alex(\varphi_x+\varphi'+\varphi_y).
  \]
  By construction, the difference between \(\varphi_x+\varphi'+\varphi_y\) and \(\varphi\) is a sum of square regions, so their Alexander gradings coincide. 
\end{proof}

\begin{definition}\label{def:true_domain}
  Let \(n\in\mathbb{N}\) with \(n>1\). 
  Suppose for \(i=1,\dots,n\), \(\tilde{\gamma}_i\) is some curve in \(\PuncturedPlane\) and \(x_i\in\HF(\gamma_{i},\gamma_{i+1})\) is an intersection point between \(\gamma_{i}\) and \(\gamma_{i+1}\), where we take indices modulo \(n\). 
  A \textbf{domain} for tuples \((\tilde{\gamma}_i)_{i=1,\dots,n}\) and \((\tilde{x}_i)_{i=1,\dots,n}\) is a domain \(\varphi\) which satisfies
  \[
  \partial\Big(\partial\varphi\cap \tilde{\gamma}_{i}\Big)=\tilde{x}_{i-1}-\tilde{x}_i,
  \] 
  where, again, indices are taken modulo \(n\).
\end{definition}

The square from Figure~\ref{fig:lem:Alex:true_domain} shows an example of such a domain for \(n=4\).

\begin{proposition}\label{prop:Alex:HFT:Euler:multiple-curves}
  With notation as in Definition~\ref{def:true_domain}, suppose the curves \(\gamma_i\) carry an absolute Alexander grading. Then, 
  \[
  \sum_{i=1}^n \Alex(x_i)=\Alex(\varphi).
  \]
\end{proposition}
\begin{proof}
  For each \(i=1,\dots, n\), choose some intersection point \(\tilde{y}_i\) of \(\tilde{\gamma}_i\) with \(P\). Then we can write \(\varphi\) as a sum of \(n\) connecting domains \(\varphi_i\) for \(\tilde{x}_i\) from \(\tilde{y}_i\) to \(\tilde{y}_{i+1}\). By Lemma~\ref{lem:Alex:HFT:Euler:two_curves}, 
  \[
  \Alex(x_i)=\Alex(y_{i+1})-\Alex(y_i)+\Alex(\varphi_i)
  \]
  for \(i=1,\dots,n\). Taking the sum over all \(n\) equations, we obtain the desired identity.
\end{proof}

\begin{figure}[bt]
	\begin{subfigure}{0.3\textwidth}
		\centering
		\(\bigonlift\)
		\caption{}\label{fig:lem:Alex:true_domain}
	\end{subfigure}
	\begin{subfigure}{0.3\textwidth}
		\centering
		\(\bigon\)
		\caption{}\label{fig:lem:Alex:bigon}
	\end{subfigure}
	\begin{subfigure}{0.3\textwidth}
		\centering
		\(\AlexLemma\)
		\caption{}\label{fig:lem:Alex:ordering}
	\end{subfigure}
	\caption{%
		(a) A once-punctured quadrilateral in \(\PuncturedPlane\) illustrating Definition~\ref{def:true_domain}, (b)
		a once-punctured bigon connecting intersection points \(x,y\in\HF(\textcolor{red}{\gamma},\textcolor{blue}{\gamma'})\) between two curves \(\textcolor{red}{\gamma}\) and \(\textcolor{blue}{\gamma'}\) in \(\FourPuncturedSphere\) as in Lemma~\ref{lem:Alex:bigon}, 
		and (c) an illustration of the proof of Lemma~\ref{lem:Alex:ordering}
	}
\end{figure}

\begin{lemma}\label{lem:Alex:bigon}
	Suppose two curves \(\gamma\) and \(\gamma'\) on \(\FourPuncturedSphere\) intersect in \(x,y\in\HF(\gamma,\gamma')\) such that there is a bigon covering a tangle end \(\TEi\) as shown in Figure~\ref{fig:lem:Alex:bigon}. Then, with notation as in Subsection~\ref{sec:review:HFT:Bigrading}, \(\Alex(y)-\Alex(x)=-2e_{\TEi}\in\GeneralAlex\), so  \(A(y)-A(x)=-\varepsilon_{\TEi}\). 
\end{lemma}

\begin{proof}
	Consider the branched double cover of the bigon in \(\PuncturedPlane\), which is illustrated in Figure~\ref{fig:lem:Alex:true_domain}. In this picture, the curves 
	\(\textcolor{red}{\tilde{\gamma}_1}\) and 
	\(\textcolor{red}{\tilde{\gamma}_3}\) are lifts of 
	\(\textcolor{red}{\gamma}\), 
	and 
	\(\textcolor{blue}{\tilde{\gamma}_2}\) and 
	\(\textcolor{blue}{\tilde{\gamma}_4}\) are lifts of 
	\(\textcolor{blue}{\gamma'}\). 
	Moreover, 
	\(\tilde{x}_1\in\textcolor{red}{\tilde{\gamma}_1}\cap\textcolor{blue}{\tilde{\gamma}_2}\)
	and 
	\(\tilde{x}_3\in\textcolor{red}{\tilde{\gamma}_3}\cap\textcolor{blue}{\tilde{\gamma}_4}\)
	are lifts of 
	\(x\in\HF(\textcolor{red}{\gamma},\textcolor{blue}{\gamma'})\), 
	and 
	\(\tilde{x}_2\in\textcolor{red}{\tilde{\gamma}_3}\cap\textcolor{blue}{\tilde{\gamma}_2}\)
	and 
	\(\tilde{x}_4\in\textcolor{red}{\tilde{\gamma}_1}\cap\textcolor{blue}{\tilde{\gamma}_4}\)
	are lifts of 
	\(y\in\HF(\textcolor{red}{\gamma},\textcolor{blue}{\gamma'})\).
	Then, by Proposition~\ref{prop:Alex:HFT:Euler:multiple-curves} and Observation~\ref{obs:Alex:reverse_order}, 
	\begin{equation*}
		4e_{\TEi}
		=
		\Alex(x_1)-\Alex(x_2)+\Alex(x_3)-\Alex(x_4)
		=
		2\Alex(x)-2\Alex(y).\qedhere
	\end{equation*}
\end{proof}

\begin{remark}\label{rem:delta:HF}
	One can show that in the situation of Lemma~\ref{lem:Alex:bigon}, the \(\delta\)-gradings of the intersection points \(x\) and \(y\) agree. Morally, this is because the \(\delta\)-grading agrees with the Maslov grading in Lagrangian Floer homology, which vanishes on once-punctured bigons. To prove this more rigorously, one can either use the techniques developed in~\cite{KWZthinness} for studying the \(\delta\)-grading, which are analogous to the techniques we are using here, or one can do this computation directly as done in the examples in~\cite[Subsection~5.2]{PQMod}.
\end{remark}

\begin{lemma}\label{lem:Alex:ordering}
  Let \(\gamma_{1}\), \(\gamma_{2}\), \(\vartheta_{1}\), and \(\vartheta_{2}\) be four rational or special curves in \(\FourPuncturedSphere\) such that \(\gamma_i\) and \(\vartheta_j\) have different slopes for all \(i,j\in\{1,2\}\). Suppose further that \(\gamma_{1}\) and \(\vartheta_{1}\) are rational. Then there exist generators \(x_{ij}\in \HF(\gamma_i,\vartheta_j)\) such that
  \[
  \Alex(x_{22})+\Alex(x_{11})=\Alex(x_{21})+\Alex(x_{12}).
  \]
  In fact, for any \(x_{22}\in\HF(\gamma_{2},\vartheta_{2})\), there exist elements \(x_{11}\), \(x_{21}\), and \(x_{12}\) with this property. 
\end{lemma}

\begin{proof}
  Let us start by lifting the curves \(\gamma_{2}\) and \(\vartheta_{2}\) to curves \(\LiftG_{2}\) and \(\LiftT_{2}\) in \(\PuncturedPlane\) such that \(\LiftG_{2}\) and \(\LiftT_{2}\) intersect at a lift \(\LiftX_{22}\) of the chosen intersection point \(x_{22}\in\HF(\gamma_{2},\vartheta_{2})\). After some homotopy, we may assume that the curves \(\LiftG_{2}\) and \(\LiftT_{2}\) are contained in \(\varepsilon\)-neighborhoods of straight lines such that the intersection \(U\) of those neighborhoods contains at most a single integer lattice point. Any straight line in \(\PuncturedPlane\) which has the same slope as \(\gamma_{1}\) and which does not intersect the integer lattice points can be regarded as a lift of \(\gamma_{1}\). The same is true, of course, for \(\vartheta_{1}\).
  In particular, we may choose lifts \(\LiftG_{1}\) and \(\LiftT_{1}\) of \(\gamma_{1}\) and \(\vartheta_{1}\), respectively, such that the intersection point \(\LiftX_{11}\) between \(\LiftT_{1}\) and \(\LiftG_{1}\) lies arbitrarily close to \(\LiftX_{22}\). 
  If \(\LiftX_{11}\) and \(\LiftX_{22}\) are sufficiently close,  
  the intersection points \(\LiftX_{ij}\in \LiftG_i\cap\LiftT_j\) define a convex quadrilateral \(\LiftX_{22}\LiftX_{12}\LiftX_{11}\LiftX_{21}\) which is entirely contained in \(U\) and does not contain any integer lattice point, as illustrated in Figure~\ref{fig:lem:Alex:ordering}. 
  This quadrilateral defines a domain for the tuples \((\LiftG_{2},\LiftT_{2},\LiftG_{1},\LiftT_{1})\) and \((\LiftX_{22},\LiftX'_{12},\LiftX_{11},\LiftX'_{21})\), where \(\LiftX'_{12}\) and \(\LiftX'_{21}\) are the points \(\LiftX_{12}\) and \(\LiftX_{21}\) regarded as elements of \(\LiftT_{2}\cap\LiftG_{1}\) and \(\LiftT_{1}\cap\LiftG_{2}\), respectively. By Proposition~\ref{prop:Alex:HFT:Euler:multiple-curves},
  \[
  \Alex(x_{22})+
  \Alex(x'_{12})+
  \Alex(x_{11})+
  \Alex(x'_{21})=0.
  \]
  By Observation~\ref{obs:Alex:reverse_order}, \(\Alex(x'_{12})=-\Alex(x_{12})\) and \(\Alex(x'_{21})=-\Alex(x_{21})\). 
  So the claim follows. 
\end{proof} 

\section{Detection of split tangles}\label{sec:detection:split}

The goal of this section is to prove Theorem~\ref{thm:detection:split:Intro}, namely $\HFT$ detects split tangles.  Recall that a tangle $T$ is \emph{split} if there exists an essential curve in $\partial B^3 \smallsetminus \partial T$ which bounds a disk in $B^3 \smallsetminus T$.  We will in fact show the following version of Theorem~\ref{thm:detection:split:Intro}, which is slightly stronger.   

\begin{theorem}\label{thm:detection:split}
  A tangle \(T\) is split if and only if \(\HFT(T)\) only contains rational components that all have the same slope. 
  If \(T\) is split, the local systems on those components are trivial. 
\end{theorem}

\begin{figure}[tb]
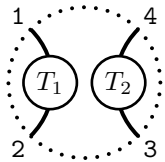

	\(\splittangle\)
	\caption{A split tangle}\label{fig:Split_tangle}
\end{figure}

\begin{proof}
	For the if-direction, suppose that \(\HFT(T)\) does not contain any special component and that all rational components have the same slope. 
  Thanks to naturality of \(\HFT(T)\) under twisting (Theorem~\ref{thm:HFT:Twisting}), it suffices to show that if this slope is \(\infty\), then \(T\) looks like Figure~\ref{fig:Split_tangle}. 
  
  Let \((M,\hat{\gamma})\) be the tangle complement equipped with a single suture parallel to the rational component, two meridional sutures around the two tangle ends \(\TEIII\) and \(\TEIV\) and one pair of oppositely oriented meridional sutures on each closed component of \(T\). 
  By \cite[Theorem~6.3]{HDsForTangles}, the intersection points of \(\HFT(T)\) with an arc on \(\FourPuncturedSphere\) of slope \(\infty\) between the two tangle ends \(\TEI\) and \(\TEII\) compute the sutured Floer homology \(\SFH(M,\hat{\gamma})\) of this balanced sutured manifold. 
  By assumption, there are no such intersection points, so this sutured Floer homology vanishes. 
  
  We can simplify \((M,\hat{\gamma})\). Let \(\gamma\) be the set of sutures obtained from \(\hat{\gamma}\) by removing the meridional sutures around the two tangle ends \(\TEIII\) and \(\TEIV\).  
  By Lemma~\ref{lem:HFT_detects_connectivity}, these two meridional sutures lie on the same tangle component. \((M,\gamma)\) is also a balanced sutured manifold. 
  Therefore, we can apply Juhász's surface decomposition formula \cite[Proposition~8.6]{SurfaceDecomposition} to show that 
  \[
  \SFH(M,\gamma)\otimes \F^2
  \cong
  \SFH(M,\hat{\gamma});
  \]
  see for example~\cite[Proof of Theorem~6.7]{HDsForTangles}. So we also obtain that \(\SFH(M,\gamma)=0\). By~\cite[Theorem~1.4]{SurfaceDecomposition}, this implies that \((M,\gamma)\) is not taut.  Since the knot Floer homology of a link in $S^3$ is non-vanishing, the K\"unneth formula for sutured Floer homology \cite[Proposition~9.15]{SFH} implies that the sutured Floer homology of the sutured exterior of $T - L$ vanishes whenever $L$ is a link contained in an embedded $B^3$ in $M$.  Consequently, the sutured exterior of $T - L$ in $M$ is not taut when $L$ is a link contained in an embedded $B^3$.  In the remainder of the proof for this direction, we will only use the tautness of $(M,\gamma)$ and hence we may restrict to the case that the exterior of $T - L$ is irreducible and show that this is split.   The splitness of $T - L$ will then imply the splitness of $T$.
  
  Then there are only two possible reasons for why \((M,\gamma)\) could not be taut: Either \(R(\gamma)\), i.e.\ the boundary of \(M\)
  minus a tubular neighborhood of the sutures \(\gamma\), is not Thurston
  norm minimizing in \(H_2(M,\gamma)\) or it is compressible. 
  Equivalently, this is true for one of the components \(R_-(\gamma)\) or \(R_+(\gamma)\) of \(R(\gamma)\), since \([R(\gamma)]=2[R_-(\gamma)]=2[R_+(\gamma)]\in H_2(M,\gamma)\). 
  Note that \(R_-(\gamma)\) and \(R_+(\gamma)\) both consist of an annulus for each closed component of \(T\) and a single genus one
  surface with one boundary component, which is parallel to
  the suture \(\gamma\). Let us denote the genus one component by \(S_\pm\). 
  If \(R_-(\gamma)\) is not Thurston norm minimizing, the suture \(\gamma\) bounds a disk, which separates the two open components of \(T\), so we are done. In the other case, when there is a compressing disk, its boundary \(\alpha\) has to lie on \(S_-\). 
  If \(\alpha\) splits \(S_-\) into two, a simple Euler characteristic argument shows that \(\alpha\) is parallel to \(\gamma\), so we obtain the same compressing disk as before. 
Otherwise, \(\alpha\) splits \(S_-\) into a three-punctured disk. The mapping class group of a once-punctured torus agrees with that of the torus, so \(\alpha\) is just some curve of rational slope. In particular, it is completely determined by its homology class in \(H_1(S_-)\), which is freely generated by the meridian of an open tangle and a curve dual to it.  Since the meridians freely generate the first homology of the tangle complement, the coefficient of the meridian in \(\alpha\) is zero. So \(\alpha\), being embedded, must be the dual curve, up to sign. This means that this open tangle component is boundary parallel, so we are also done in this case.
  
  The only-if-direction follows more easily: \((M,\gamma)\) is not taut, so \(\SFH(M,\gamma)=0\), so also \(\SFH(M,\hat{\gamma})=0\).  Therefore, up to twisting, \(\HFT(T)\) does not intersect the arcs of slope \(\infty\) between tangle ends, so the only components can be rationals of slope \(\infty\).  It remains to prove that the local systems are trivial.  Let $r_1,\ldots, r_n$ be the components of $\HFT(T)$, all of slope $\infty$, and let their local systems be represented by the matrices $X_1,\ldots,X_n$ of dimensions $k_1,\ldots,k_n$, respectively.  Then, by the gluing theorem in conjunction with \cite[Theorem~4.45]{PQMod}, 
\[
\dim \HFKhat(T(0)) = \sum_{i=1}^n k_i , \quad \dim \HFKhat(T(\infty)) = 2 \sum_{i=1}^n \dim \ker (X_i - \id).
\]
For \(i=1,2\), let \(L_i\) be the link defined by closing the two-ended tangle \(T_i\) from Figure~\ref{fig:Split_tangle}. Then $T(0)$ is the connected sum \(L_1 \# L_2\) and $T(\infty)$  is the split link $L_1 \amalg L_2$. 
By the Künneth formulas for knot Floer homology, $2 \dim \HFKhat(L_1 \# L_2) = \dim \HFKhat(L_1 \amalg L_2)$ \cite [Equations (5) and (6)]{HFK}.  Therefore, it follows that $X_i = \id$ for each $i$, i.e.\ the local systems are all trivial.
\end{proof}

\begin{remark}
An alternate strategy to the proof of Theorem~\ref{thm:detection:split} involves proving a K\"unneth formula for $\HFT$: For any tangle \(T\) and any link \(L\), 
  \[
 	\HFT(T\amalg L)\cong \Big(V'\otimes\HFKhat(L)\Big)\otimes \HFT(T).
  \] 
  Moreover, if \(T_1\) and \(T_2\) are two-ended tangles and \(T\) is their disjoint union as shown in Figure~\ref{fig:Split_tangle}, then
	\[
	\HFT(T)
	\cong
	\Big(\HFKhat(L_1)\otimes\HFKhat(L_2)\Big)\otimes\HFT(Q_\infty),
	\]
	where \(L_1\) and \(L_2\) are the links obtained as the closures of \(T_1\) and \(T_2\), respectively. 
Since the invariant \(\HFT\) is essentially a multi-pointed Heegaard Floer theory \cite[Definition~2.16]{PQMod}, these formulas follow from the same arguments as the Künneth formulas for link Floer homology \cite[Section~11]{OSHFL}.  Alternatively, they can also be proved using nice Heegaard diagrams, similar to the argument used in~\cite[Proof of Theorem~4.2]{PQSym}. 
\end{remark}

\section{Pairing calculations}\label{sec:pairings}

Throughout this section, let us fix some \(\alpha\in\Z^{>0}\). We will study the Lagrangian Floer homology of the special curve \(\textcolor{red}{\Special_\alpha(0;\TEIV,\TEI)}\) and other special and rational curves, restricting ourselves to the relative univariate Alexander and \(\delta\)-grading. 
Before doing so, we introduce some notation.

\begin{definition}
Given some \(k\in\Z^{>0}\), let \(\Contiguous{k}\) denote the contiguous, absolutely bigraded vector space of dimension \(k\), supported in Alexander gradings \(1,\dots,k\) and \(\delta\)-grading 0.
\end{definition}

\begin{lemma}\label{lem:computation:specials}
	Let \(\beta\in\Z^{>0}\) with \(\beta\leq\alpha\).  Suppose \(\HFr(\textcolor{red}{\Special_\alpha(0;\TEIV,\TEI)},\textcolor{blue}{\Special_\beta(0;\TEIV,\TEI)})\) is skeletal. Then 
	\[
	\HFr(\textcolor{red}{\Special_\alpha(0;\TEIV,\TEI)},\textcolor{blue}{\Special_\beta(0;\TEIV,\TEI)})
	=
	\delta^0 t^{-\alpha} \Contiguous{2\beta}
	\oplus
	\delta^{\pm1} t^\alpha \Contiguous{2\beta}
	\]
	as relatively bigraded vector spaces. 
	Moreover, the punctures \(\TEI\) and \(\TEIV\) of \(\FourPuncturedSphere\) are oriented in the same direction. 
\end{lemma}

\begin{lemma}\label{lem:computation:mixed}
	Let \(n\in\Z\). Suppose \(\HFr(\textcolor{red}{\Special_\alpha(0;\TEIV,\TEI)},\textcolor{blue}{\Rational(\tfrac{1}{n})})\) is skeletal. Then 
	\[
	\HFr(\textcolor{red}{\Special_\alpha(0;\TEIV,\TEI)},\textcolor{blue}{\Rational(\tfrac{1}{n})})
	=
	\Contiguous{2\alpha}
	\]
	as relatively bigraded vector spaces. 
\end{lemma}

\begin{remark}
	In the graphical notation from Remark~\ref{rem:gaps}, the two vector spaces from Lemmas~\ref{lem:computation:specials} and~\ref{lem:computation:mixed} look as follows:
	\[\setlength{\tabcolsep}{10pt}
		\begin{tabular}{cc}
			\(\HFr(\textcolor{red}{\Special_\alpha(0;\TEIV,\TEI)},\textcolor{blue}{\Special_\beta(0;\TEIV,\TEI)})\)
			&
			\(\HFr(\textcolor{red}{\Special_\alpha(0;\TEIV,\TEI)},\textcolor{blue}{\Rational(\tfrac{1}{n})})\)
			\\
			\(
			\vc{%
				\begin{tikzpicture}[scale=0.5]
				\draw[->] (0,0) node [left] {\phantom{\(A\)}} -- (16,0) node[right]{\(A\)};
				\draw (1,-0.25) -- (1,0.25) node[above] {\(\bullet\)};
				\draw (2,-0.25) -- (2,0.25) node[above] {\(\bullet\)};
				\draw (3,0.25) node [above] {\(\cdots\)};
				\draw (4,-0.25) -- (4,0.25) node[above] {\(\bullet\)};
				\draw (5,-0.25) -- (5,0.25) node[above] {\(\bullet\)};
				\draw [decorate,decoration={brace}] (5.25,-0.4) -- (0.75,-0.4) node[midway, below] {\(\scriptstyle 2\beta\)};
				\draw (6,-0.25) -- (6,0.25);
				\draw (7,-0.25) -- (7,0.25);
				\draw (8,0.25) node [above] {\(\cdots\)};
				\draw (9,-0.25) -- (9,0.25);
				\draw (10,-0.25) -- (10,0.25);
				\draw [decorate,decoration={brace}] (10.25,-0.4) -- (5.75,-0.4) node[midway, below] {\(\scriptstyle 2(\alpha-\beta)\)};
				\draw (11,-0.25) -- (11,0.25) node[above] {\(\circ\)};
				\draw (12,-0.25) -- (12,0.25) node[above] {\(\circ\)};
				\draw (13,0.25) node [above] {\(\cdots\)};
				\draw (14,-0.25) -- (14,0.25) node[above] {\(\circ\)};
				\draw (15,-0.25) -- (15,0.25) node[above] {\(\circ\)};
				\draw [decorate,decoration={brace}] (15.25,-0.4) -- (10.75,-0.4) node[midway, below] {\(\scriptstyle 2\beta\)};
		\end{tikzpicture}}
			\)
			&
			\(
			\vc{%
				\begin{tikzpicture}[scale=0.5]
				\draw[->] (0,0) node [left] {\phantom{\(A\)}} -- (6,0) node[right]{\(A\)};
				\draw (1,-0.25) -- (1,0.25) node[above] {\(\bullet\)};
				\draw (2,-0.25) -- (2,0.25) node[above] {\(\bullet\)};
				\draw (3,0.25) node [above] {\(\cdots\)};
				\draw (4,-0.25) -- (4,0.25) node[above] {\(\bullet\)};
				\draw (5,-0.25) -- (5,0.25) node[above] {\(\bullet\)};
				\draw [decorate,decoration={brace}] (5.25,-0.4) -- (0.75,-0.4) node[midway, below] {\(\scriptstyle 2\alpha\)};
		\end{tikzpicture}}
			\)
		\end{tabular}
	\]
	In the first vector space, the \(\delta\)-grading of the generators \(\bullet\) is constant, the same is true for the \(\delta\)-grading of the generators \(\circ\), and the difference between these two \(\delta\)-gradings is \(\pm1\).  
	The sign is immaterial for the proof of the main theorem. As will be apparent from the proof below, this sign only depends on the orientation of the punctures \(\TEI\) and \(\TEIV\).
\end{remark}

\begin{figure}[t]
	\(\PairingIrrIrrParallel\)
	\caption{Calculation of \(\HFr(\textcolor{red}{\Special_\alpha(0;\TEIV,\TEI)},\textcolor{blue}{\Special_\beta(0;\TEIV,\TEI)})\) for the proof of Lemma~\ref{lem:computation:specials}}\label{fig:computation:specials}
\end{figure}

\begin{proof}[Proof of Lemma~\ref{lem:computation:specials}]
	Figure~\ref{fig:computation:specials} shows the diagram from which we will compute the vector space \(\HFr(\textcolor{red}{\Special_\alpha(0;\TEIV,\TEI)},\textcolor{blue}{\Special_\beta(0;\TEIV,\TEI)})\) in the case that \(\beta\) is even. 
	If \(\beta\) is odd, there is a very similar diagram that can be obtained by rotating the outer part by \(\pi\) in the drawing plane such that the outermost blue curve segments lie on the right. 
	We claim that \(\HF(\textcolor{red}{\Special_\alpha(0;\TEIV,\TEI)},\textcolor{blue}{\Special_\beta(0;\TEIV,\TEI)})\) is freely generated by the intersection points in this diagram.

	For each \(\varepsilon=1,\dots,\beta\), \(z\in\{x,y\}\) and \(k\in\{1,2\}\), there is a bigon from \(z_{k1}^\varepsilon\) to \(z_{k2}^\varepsilon\) covering the puncture \(\TEI\) once. 
	So by Lemma~\ref{lem:Alex:bigon} and Remark~\ref{rem:delta:HF}, each of these pairs of generators sits in consecutive Alexander gradings and a single \(\delta\)-grading.
	To compute \(\HFr(\textcolor{red}{\Special_\alpha(0;\TEIV,\TEI)},\textcolor{blue}{\Special_\beta(0;\TEIV,\TEI)})\), it therefore suffices to restrict our attention to the intersection points \(z_{k1}^\varepsilon\), where \(\varepsilon=1,\dots,\beta\), \(z\in\{x,y\}\) and \(k\in\{1,2\}\). The following diagram shows these generators arranged in the 2-dimensional plane according to their Alexander gradings:
	\[
	\begin{tikzpicture}[xscale=0.85,yscale=0.425,decoration=brace]
	\draw [->] (-9.25,-1.5) -- (-7.5,2) node [above] {\(\Z e_{\TEIV}\subset\GeneralAlex\)};
	\draw [->] (-9.25,-0.5) -- (-7.5,-4) node [below] {\(\Z e_{\TEI}\subset\GeneralAlex\)};
	
	\node (X1_1) at (-8,-1) {$x_{11}^1$};
	\node (X2_1) at (-7,+1) {$x_{21}^1$};
	\node (X1_2) at (-6,-1) {$x_{11}^2$};
	\node (X2_2) at (-5,+1) {$x_{21}^2$};
	\node (X1_3) at (-4,-1) {};
	\node (X2_3) at (-3,+1) {};
	\node (X1_b) at (-2,-1) {$x_{11}^\beta$};
	\node (X2_b) at (-1,+1) {$x_{21}^\beta$};
	
	\node (Y2_1) at (8,+1) {$y_{21}^1$};
	\node (Y1_1) at (7,-1) {$y_{11}^1$};
	\node (Y2_2) at (6,+1) {$y_{21}^2$};
	\node (Y1_2) at (5,-1) {$y_{11}^2$};
	\node (Y2_3) at (4,+1) {};
	\node (Y1_3) at (3,-1) {};
	\node (Y2_b) at (2,+1) {$y_{21}^\beta$};
	\node (Y1_b) at (1,-1) {$y_{11}^\beta$};
	
	\draw (X1_1) -- (X2_1);
	\draw (X1_2) -- (X2_2);
	\draw (X1_b) -- (X2_b);
	\draw (X1_1) -- (X1_2);
	\draw (X1_2) -- (X1_3);
	\draw (X1_3) -- (X1_b);
	
	\draw (Y1_1) -- (Y2_1);
	\draw (Y1_2) -- (Y2_2);
	\draw (Y1_b) -- (Y2_b);
	\draw (Y1_1) -- (Y1_2);
	\draw (Y1_2) -- (Y1_3);
	\draw (Y1_3) -- (Y1_b);
	
	\draw [dashed] (X1_b) -- (Y1_b);
	
	\node (x) at (-0.5,-4) {};
	\draw [<->] (x) -- (Y1_b) node [pos=0.5, below right] {\footnotesize\((\alpha-\beta+1)\)};
	\draw [<->] (x) -- (X1_b) node [pos=0.5, below left] {\footnotesize\((\alpha-\beta+1)\)};
	\end{tikzpicture}
	\]
	The lines connecting these generators indicate certain domains from which we compute the relative positions of all generators: For each \(\varepsilon=1,\dots,\beta\) and \(z\in\{x,y\}\), there is a bigon from \(z_{11}^\varepsilon\) to \(z_{21}^\varepsilon\) covering the puncture \(\TEIV\) once; these bigons correspond to the diagonal lines and preserve the \(\delta\)-grading. For every \(\varepsilon=1,\dots,\beta-1\), there are domains covering both punctures \(\TEI\) and \(\TEIV\) connecting \(x_{11}^\varepsilon\) to \(x_{11}^{\varepsilon+1}\) and \(y_{11}^{\varepsilon+1}\) to \(x_{11}^{\varepsilon}\); these correspond to the solid horizontal lines and also preserve the \(\delta\)-grading.  Finally, there is a domain from \(x_{11}^\beta\) to \(y_{11}^{\beta}\) covering the punctures \(\TEI\) and \(\TEIV\) \((\alpha-\beta+1)\) times; it corresponds to the dashed horizontal line in the centre of the diagram above and changes the \(\delta\)-grading by one.
	Note that the other half of intersection points, i.e.\ the generators \(\{z_{k2}^\varepsilon\}\), is obtained by shifting the generators \(\{z_{k1}^\varepsilon\}\) by one unit in the \(e_{\TEI}\)-direction.  In particular, the only pair of generators \((x_{ij}^\varepsilon,y_{k\ell}^{\nu})\) that share the same Alexander grading is \((x_{22}^\beta,y_{11}^\beta)\), and this only if \(\alpha=\beta\).
	 
	There are no immersed annuli with non-negative multiplicities in this diagram, which can be seen, for example, by considering the induced boundary orientation on \(\textcolor{blue}{\Special_\beta(0;\TEIV,\TEI)}\) near the puncture \(\TEI\). Therefore, the intersection points generate the complex whose homology is \(\HF(\textcolor{red}{\Special_\alpha(0;\TEIV,\TEI)},\textcolor{blue}{\Special_\beta(0;\TEIV,\TEI)})\).  To compute the differentials, observe that the \(\delta\)-grading is constant on the intersection points \(x_{k\ell}^\varepsilon\), and so is the \(\delta\)-grading on \(y_{k\ell}^\varepsilon\).
	So for grading reasons, only the generators \(x_{22}^\beta\) and \(y_{11}^\beta\) can be connected by a potentially non-vanishing differential, and this only if \(\alpha=\beta\). In this case, there are in fact two bigons between these two intersection points, so the differential always vanishes.  
	(In fact, for \(\alpha=\beta\), the generator \(x_{22}^\beta\) corresponds to the identity morphism of \(\textcolor{red}{\Special_\alpha(0;\TEIV,\TEI)}=\textcolor{blue}{\Special_\beta(0;\TEIV,\TEI)}\) in the Fukaya category.)
	
	Suppose the punctures \(\TEI\) and \(\TEIV\) are oriented differently, i.e.\ we quotient \(\GeneralAlex\) by \(e_{\TEI}+e_{\TEIV}\) when passing to the univariate Alexander grading. 
	Then these generators are concentrated in only two Alexander gradings, so \(\HFr(\textcolor{red}{\Special_\alpha(0;\TEIV,\TEI)},\textcolor{blue}{\Special_\beta(0;\TEIV,\TEI)})\) is not skeletal. 
	Hence the punctures are oriented in the same direction, i.e.\ we quotient \(\GeneralAlex\) by \(e_{\TEI}-e_{\TEIV}\). 
	In this case,  \(\HFr(\textcolor{red}{\Special_\alpha(0;\TEIV,\TEI)},\textcolor{blue}{\Special_\beta(0;\TEIV,\TEI)})\) has the desired form.  
\end{proof}

\begin{proof}[Proof of Lemma~\ref{lem:computation:mixed}]
	Figure~\ref{fig:computation:mixed} shows all intersection points \(x_\varepsilon\), \(\varepsilon=1,\dots,4\alpha\), between these curves in minimal position.
	Observe that there is a 
	bigon from \(x_{2\alpha+1-\varepsilon}\) to \(x_{2\alpha+\varepsilon}\) 
	covering the puncture \(\TEI\) for all \(\varepsilon=1,\dots,2\alpha\).
	Therefore \(\HFr(\textcolor{red}{\Special_\alpha(0;\TEIV,\TEI)},\textcolor{blue}{\Rational(\tfrac{1}{n})})\) is isomorphic to the relatively bigraded vector space generated by \(\{x_{2\alpha+\varepsilon}\}_{\varepsilon=1,\dots,2\alpha}\).  
	We claim that all generators have the same \(\delta\)-grading and their Alexander grading is as follows:	
	\[
	\begin{tikzpicture}[xscale=0.85,yscale=0.425,decoration=brace]
	\draw [->] (-9.25,-1.5) -- (-7.5,2) node [above] {\(\Z e_{\TEIV}\subset\GeneralAlex\)};
	\draw [->] (-9.25,-0.5) -- (-7.5,-4) node [below] {\(\Z e_{\TEI}\subset\GeneralAlex\)};
	
	\node (X8) at (-8,-1) {$x_{4\alpha}$};
	\node (X7) at (-7,+1) {$x_{4\alpha-1}$};
	\node (X6) at (-6,-1) {$x_{4\alpha-2}$};
	\node (X5) at (-5,+1) {$x_{4\alpha-3}$};
	\node (X4) at (-4,-1) {};
	\node (X3) at (-3,+1) {};
	\node (X2) at (-2,-1) {$x_{2\alpha+2}$};
	\node (X1) at (-1,+1) {$x_{2\alpha+1}$};
	
	\draw (X1) -- (X2);
	\draw [dashed] (X2) -- (X3);
	\draw (X3) -- (X4);
	\draw [dashed] (X4) -- (X5);
	\draw (X5) -- (X6);
	\draw [dashed] (X6) -- (X7);
	\draw (X7) -- (X8);
	\end{tikzpicture}
	\]
	As in the proof of Lemma~\ref{lem:computation:specials}, the solid and dashed lines between the generators in this picture correspond to certain domains: The solid lines represent bigons from \(x_{2\alpha+2\varepsilon}\) to \(x_{2\alpha+2\varepsilon-1}\) 
	covering the puncture \(\TEIV\), where \(\varepsilon=1,\dots,\alpha\). The dashed lines correspond to the bigons from \(x_{2\alpha+2\varepsilon+1}\) to \(x_{2\alpha+2\varepsilon}\) 
	covering the puncture \(\TEI\), where \(\varepsilon=1,\dots,\alpha-1\). 
	Note that all generators sit in the same \(\delta\)-grading by Remark~\ref{rem:delta:HF}.
	Now, if the two punctures \(\TEIV\) and \(\TEI\) are oppositely oriented and \(\alpha>1\), \(\HFr(\textcolor{red}{\Special_\alpha(0;\TEIV,\TEI)},\textcolor{blue}{\Rational(\tfrac{1}{n})})\) is not skeletal. Otherwise, the vector space is contiguous. 
\end{proof}

\begin{lemma}\label{lem:reduction:specials}
	For any \(\beta\in\Z^{>0}\), \(s\in\QPI\smallsetminus\{0\}\), and \(\TEi,\TEj\in\{\TEI,\TEII,\TEIII,\TEIV\}\),  \(\HFr(\textcolor{red}{\Special_\alpha(0;\TEIV,\TEI)},\textcolor{blue}{\Special_\beta(s;\TEi,\TEj)})\) is not skeletal.
\end{lemma}

\begin{lemma}\label{lem:reduction:mixed}
	Let \(s\in\QPI\smallsetminus\left(\{0\}\cup\{\tfrac{1}{n}\mid n\in\Z\}\right)\). Then  \(\HFr(\textcolor{red}{\Special_\alpha(0;\TEIV,\TEI)},\textcolor{blue}{\Rational(s)})\) is not skeletal.
\end{lemma}

\begin{figure}[tb]
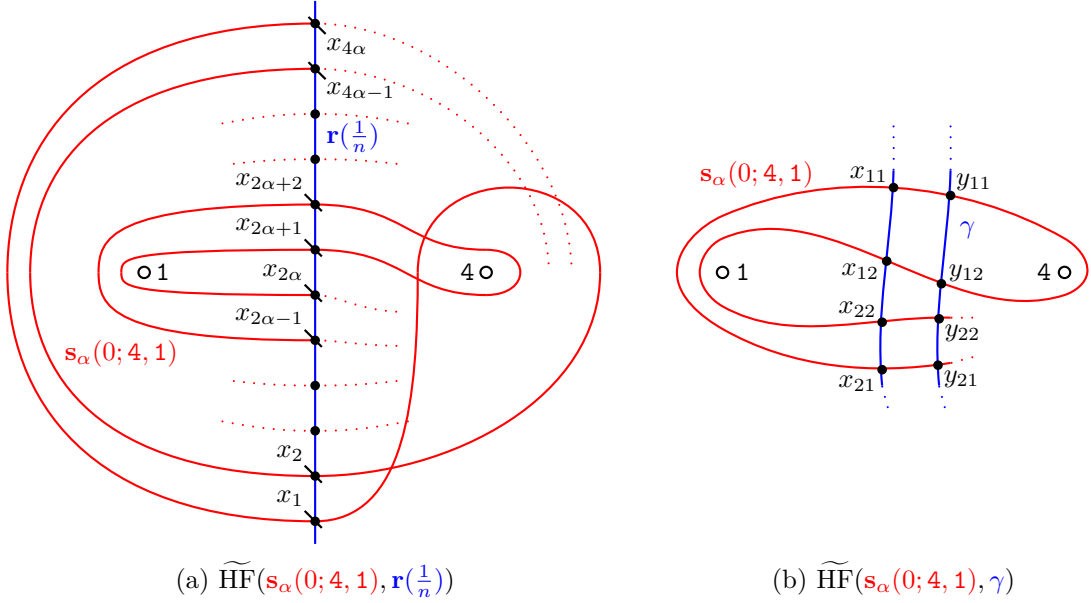

	\centering
	\begin{subfigure}{0.54\textwidth}
		\centering
		\(\PairingIrrRat\)
		\caption{\(\HFr(\textcolor{red}{\Special_\alpha(0;\TEIV,\TEI)},\textcolor{blue}{\Rational(\tfrac{1}{n})})\)}\label{fig:computation:mixed}
	\end{subfigure}
	\begin{subfigure}{0.4\textwidth}
		\centering
		\(\PairingIrrIrrNonparallel\)
		\caption{\(\HFr(\textcolor{red}{\Special_\alpha(0;\TEIV,\TEI)},\textcolor{blue}{\gamma})\)}\label{fig:reduction}
	\end{subfigure}
	\caption{Computations for the proof of (a) Lemma~\ref{lem:computation:mixed} and (b)  Lemmas~\ref{lem:reduction:specials} and~\ref{lem:reduction:mixed} }\label{fig:specials}
\end{figure}

\begin{proof}[Proof of Lemmas~\ref{lem:reduction:specials} and~\ref{lem:reduction:mixed}]
	Let \(\textcolor{blue}{\gamma}\) be either equal to \(\textcolor{blue}{\Special_\beta(s;\TEi,\TEj)}\) from Lemma~\ref{lem:reduction:specials} or  \(\textcolor{blue}{\Rational(s)}\) from Lemma~\ref{lem:reduction:mixed}. Let us put \(\textcolor{red}{\Special_\alpha(0;\TEIV,\TEI)}\) and \(\textcolor{blue}{\gamma}\) into minimal position. We claim that there are portions of \(\textcolor{red}{\Special_\alpha(0;\TEIV,\TEI)}\) and \(\textcolor{blue}{\gamma}\) that look like Figure~\ref{fig:reduction}, where the two curve segments of \(\textcolor{blue}{\gamma}\) stay parallel until they are joined up near some puncture. 
	To see this, let \(a\) be the horizontal arc connecting the punctures \(\TEI\) and \(\TEIV\). 
	If \(\textcolor{blue}{\gamma}\) is special, it is a figure-eight curve which sits in the neighborhood of some immersed arc connecting a puncture to itself. It suffices to see that this arc intersects \(a\) non-trivially. Suppose it does not. Then its slope is either 0 or \(\tfrac{1}{n}\) for some \(n\in\Z\). The first case contradicts our assumption on \(s\). In the second case, either \(\TEi\not\in\{\TEI,\TEIV\}\) or \(\TEj\not\in\{\TEI,\TEIV\}\). Without loss of generality, let us assume the former. Then we may reinterpret \(\textcolor{blue}{\gamma}\) as a figure-eight curve which sits in the neighborhood of a different immersed arc of the same slope as the first, but one which starts and ends at the puncture \(\TEi\neq\TEI,\TEIV\). This new arc intersects \(a\) non-trivially.
	A similar argument applies if \(\textcolor{blue}{\gamma}\) is rational. In this case, we can think of \(\textcolor{blue}{\gamma}\) as the boundary of a tubular neighborhood of an embedded arc, which connects a pair of distinct punctures. We may choose this arc such that at least one of its ends is different from \(\TEI\) and \(\TEIV\). Since in this case, \(s\neq\tfrac{1}{n}\) for all \(n\in\Z\) and \(s\neq0\) by assumption, this arc also intersects \(a\) non-trivially.
	%

	We now compute the Lagrangian Floer homology from this picture. 
	For any \(k,\ell\in\{1,2\}\), there are bigons from \(x_{k\ell}\) to \(y_{k\ell}\) covering a single puncture. So \(\HFr(\textcolor{red}{\Special_\alpha(0;\TEIV,\TEI)},\textcolor{blue}{\gamma})\) contains a summand which is isomorphic to a vector space generated by the four intersection points \(x_{k\ell}\), \(k,\ell\in\{1,2\}\). One can easily see that this summand is (relatively) bigraded isomorphic to \(V^{\otimes2}\), which is not skeletal. 
\end{proof}

\section{Proof of the Main Theorem}\label{sec:proof-main-thm}

Throughout this section, we fix a decomposition of a knot \(K=T_1\cup T_2\) into two four-ended tangles. 
We will write \(\Gammai\coloneqq\mr(\HFT(T_1))\) and \(\Gammaii\coloneqq\HFT(T_2)\). 
The proof of the Main Theorem is based on the following two propositions together with Theorem~\ref{thm:detection:split}.

\begin{proposition}\label{prop:NoLspaceKnot:IrrIrr}
	Suppose both \(\Gammai\) and \(\Gammaii\) contain a special component. Then \(K\) is not an L-space knot. 
\end{proposition}

\begin{proposition}\label{prop:NoLspaceKnot:IrrRat}
	Suppose \(T_1\) is not rational and \(\Gammaii\) consists entirely of multiple rational components not all of the same slope. Then \(K\) is not an L-space knot.  
\end{proposition}

\begin{proof}[Proof of the Main Theorem]
	Suppose that an L-space knot $K$ can be expressed in terms of a tangle decomposition $K = T_1\cup T_2$ along an essential Conway sphere, where \(T_1\) and \(T_2\) are four-ended tangles. 
	Then by Proposition~\ref{prop:NoLspaceKnot:IrrIrr}, either \(\Gammai\) or  \(\Gammaii\) contains only rational components. 
	Without loss of generality, we may assume \(\Gammaii\) contains only rational components.
	Because the Conway sphere is assumed to be essential, $T_1$ is not a rational tangle, so by Proposition~\ref{prop:NoLspaceKnot:IrrRat}, the components of $\Gammaii$ must all have the same slope. 
	By Theorem~\ref{thm:detection:split}, the corresponding tangle $T_2$ is split, which contradicts the assumption that the Conway sphere is essential.
\end{proof}

The rest of this section is devoted to the proofs of Propositions~\ref{prop:NoLspaceKnot:IrrIrr} and~\ref{prop:NoLspaceKnot:IrrRat}.  We begin with a corollary of Lemma~\ref{lem:lspacestructure} which will be central to both proofs.

\begin{lemma}\label{lem:lspacepinch}
	Let \(T_1, T_2\) be tangles such that \(T_1 \cup T_2\) is an L-space knot. For \(i=1,2\), suppose \(s_i\in\QPI\) is the slope of a rational component of \(\HFT(T_i)\).  Then, \(T_1 \cup Q_{s_2}\) and \(Q_{s_1} \cup T_2\) are also L-space knots.
\end{lemma}
\begin{proof}
	By symmetry of the tangle decomposition, it suffices to show this for \(i=1\). 
	It follows from the pairing theorem that \(\HFKhat(Q_{s_1} \cup T_2)\) is a summand of \(\HFKhat(T_1 \cup T_2)\), up to an absolute shift of Maslov and Alexander gradings.  Since \(T_1 \cup T_2\) satisfies the hypotheses of Lemma~\ref{lem:lspacestructure} by Corollary~\ref{cor:lspace-ordered-short}, the same is true of \(Q_{s_1} \cup T_2\), and hence it is an L-space knot.  
\end{proof}
 
\begin{corollary}\label{cor:Lspacepinch}
	Suppose \(K\) is an L-space knot. Then, for any rational component \(\textcolor{red}{\gamma}\) of \(\Gammai\), the subspace \(W=\HFr(\textcolor{red}{\gamma},\Gammaii)\) of \(\HFr(\Gammai,\Gammaii)\) satisfies the condition from Lemma~\ref{lem:convex}. 
 	The same holds for the subspace \(\HFr(\Gammai,\textcolor{blue}{\gamma'})\) for any rational component \(\textcolor{blue}{\gamma'}\) of \(\Gammaii\).
 	\qed
\end{corollary}
 
\begin{proof}[Proof of Proposition~\ref{prop:NoLspaceKnot:IrrIrr}]
	Suppose for contradiction that \(K\) is an L-space knot. 
	By naturality under twisting (Theorem~\ref{thm:HFT:Twisting}), we may assume without loss of generality that \(\Gammai\) contains a special component \(\textcolor{red}{\Special_\alpha(0;\TEIV,\TEI)}\) for some \(\alpha\in\Z^{>0}\). 
	By Corollary~\ref{cor:lspace-ordered-short}, \(K\) is skeletal, so by Lemma~\ref{lem:reduction:specials}, the special components of \(\Gammaii\) all have slope 0. Then by conjugation symmetry (Theorem~\ref{thm:Conjugation}), \(\Gammaii\) in particular contains a component \(\textcolor{blue}{\Special_\beta(0;\TEIV,\TEI)}\) for some \(\beta\in\Z^{>0}\). 
	Repeating this argument with reversed roles of \(T_1\) and \(T_2\), we see that also the slope of any special component of \(\Gammai\) vanishes. 
	Moreover, we may assume without loss of generality that \(\beta\leq\alpha\) by interchanging \(T_1\) and \(T_2\) if necessary.  
	
	Now, by Lemma~\ref{lem:computation:specials} in conjunction with the fact that \(K\) is skeletal,
	\[
	\HFr(\textcolor{red}{\Special_\alpha(0;\TEIV,\TEI)},\textcolor{blue}{\Special_\beta(0;\TEIV,\TEI)})
	=
	\delta^0 t^{-\alpha} \Contiguous{2\beta}
	\oplus
	\delta^{\pm1} t^\alpha \Contiguous{2\beta}
	\]
	up to some overall shift in bigrading.
	Moreover, since the tangle ends \(\TEI\) and \(\TEIV\) are oriented in the same direction, the slope of any rational component of \(\Gammai\) or \(\Gammaii\) is non-zero. 
	Let \(g_1\) be the generator in maximal Alexander grading of the first summand and \(g_2\) the generator in minimal Alexander grading of the second summand. Then \(A(g_2)-A(g_1)=2(\alpha-\beta)+1\) is odd and so is \(\delta(g_2)-\delta(g_1)=\pm1\), hence, \(M(g_2)-M(g_1)\) is even.  By Theorem~\ref{thm:lspace-ordered-stronger}, the number of generators whose Maslov gradings lie strictly between \(M(g_1)\) and \(M(g_2)\) is odd. Equivalently, the number of generators whose Alexander gradings lie strictly between \(A(g_1)\) and \(A(g_2)\) is odd.  
	We claim that this implies the existence of rational components \(\textcolor{red}{\gamma}\) of \(\Gammai\) and \(\textcolor{blue}{\gamma'}\) of \(\Gammaii\) with the property that \(\HFr(\textcolor{red}{\gamma},\textcolor{blue}{\gamma'})\) is a non-zero contiguous vector space which is supported in Alexander gradings that lie strictly between \(A(g_1)\) and \(A(g_2)\).
	Indeed, by Lemma~\ref{lem:computation:specials}, a skeletal \(\HFr\)-pairing of any two special components of  \(\Gammai\) and \(\Gammaii\) consists of two contiguous even-dimensional summands. Moreover, by Lemma~\ref{lem:reduction:mixed}, the \(\HFr\)-pairing between any rational and any special component is skeletal only if the slope of the rational component is 0 or \(\tfrac{1}{n}\) for some integer \(n\). However, as noted above, the slope of the rational component cannot be 0. If the slope is \(\tfrac{1}{n}\), we can apply Lemma~\ref{lem:computation:mixed} to deduce that the pairing is contiguous and even-dimensional.
	Finally, by Lemma~\ref{lem:2bridge}, a skeletal \(\HFr\)-pairing of any two rationals is contiguous. 
	
	In summary, we have identified two non-zero summands of \(\HFr(\Gammai,\Gammaii)\), which in the same graphical notation as in Remark~\ref{rem:gaps} look as follows:
	\[
		\vc{%
			\begin{tikzpicture}[scale=0.5]
			\draw[->] (0,0) node [left] {\phantom{\(A\)}} -- (16,0) node[right]{\(A\)};
			\draw (1,-0.25) -- (1,0.25) node[above] {\(\circ\)};
			\draw (2,-0.25) -- (2,0.25) node[above] {\(\circ\)};
			\draw (3,0.25) node [above] {\(\cdots\)};
			\draw (4,-0.25) -- (4,0.25) node[above] {\(\circ\)};
			\draw (5,-0.25) -- (5,0.25) node[above] {\(\circ\)} ++(0,0.5) node[above] {\(\scriptstyle g_1\)};
			\draw [decorate,decoration={brace}] (5.25,-0.4) -- (0.75,-0.4) node[midway, below] {\(\scriptstyle 2\beta\)};
			\draw (6,-0.25) -- (6,0.25);
			\draw (7,-0.25) -- (7,0.25);
			\draw [<-] (8,0.75) .. controls (8,1.5) and (7,2.5) .. (8.5,2.5) node [right] {\(\HFr(\textcolor{red}{\gamma},\textcolor{blue}{\gamma'})\)};
			\draw (9,-0.25) -- (9,0.25);
			\draw (10,-0.25) -- (10,0.25);
			\draw [decorate,decoration={brace}] (10.25,-0.4) -- (5.75,-0.4) node[midway, below] {\(\scriptstyle 2(\alpha-\beta)\)};
			\draw (11,-0.25) -- (11,0.25) node[above] {\(\circ\)} ++(0,0.5) node[above] {\(\scriptstyle g_2\)};
			\draw (12,-0.25) -- (12,0.25) node[above] {\(\circ\)};
			\draw (13,0.25) node [above] {\(\cdots\)};
			\draw (14,-0.25) -- (14,0.25) node[above] {\(\circ\)};
			\draw (15,-0.25) -- (15,0.25) node[above] {\(\circ\)};
			\draw [decorate,decoration={brace}] (15.25,-0.4) -- (10.75,-0.4) node[midway, below] {\(\scriptstyle 2\beta\)};
			\end{tikzpicture}}
	\]
	where \(\circ\in\HFr(\textcolor{red}{\Special_\alpha(0;\TEIV,\TEI)},\textcolor{blue}{\Special_\beta(0;\TEIV,\TEI)})\).
	Now observe that  \(\HFr(\textcolor{red}{\Special_\alpha(0;\TEIV,\TEI)},\textcolor{blue}{\gamma'})\) is non-zero, since the slope of \(\textcolor{blue}{\gamma'}\) is non-zero as noted above. Moreover, since this summand is skeletal, it is contiguous of dimension \(2\alpha>2(\alpha-\beta)\) by Lemmas~\ref{lem:reduction:mixed} and~\ref{lem:computation:mixed}.  In particular, it is supported in Alexander gradings that are either strictly bigger than the maximum Alexander grading or strictly smaller than the minimal Alexander grading of \(\HFr(\textcolor{red}{\Special_\alpha(0;\TEIV,\TEI)},\textcolor{blue}{\Special_\beta(0;\TEIV,\TEI)})\).  
	Using the summands \(\HFr(\textcolor{red}{\gamma},\textcolor{blue}{\gamma'})\), \(\HFr(\textcolor{red}{\Special_\alpha(0;\TEIV,\TEI)},\textcolor{blue}{\Special_\beta(0;\TEIV,\TEI)})\), and \(\HFr(\textcolor{red}{\Special_\alpha(0;\TEIV,\TEI)},\textcolor{blue}{\gamma'})\), we see that the conditions of Lemma~\ref{lem:convex} are satisfied for the subspace \(W=\HFr(\Gammai,\textcolor{blue}{\gamma'})\). This violates Corollary~\ref{cor:Lspacepinch}.  
\end{proof}

\begin{remark}
	Note that if \(\alpha=\beta\) in the proof above, there cannot be any generator whose Alexander grading lies strictly between \(A(g_1)\) and \(A(g_2)\) for trivial reasons, since \(A(g_2)-A(g_1)=1\). There are currently no known examples of non-split tangles \(T\) such that \(\HFT(T)\) does not contain a special component of length \(1\). 
\end{remark}

It remains to prove Proposition~\ref{prop:NoLspaceKnot:IrrRat}. From now on, we will always assume the hypotheses of this proposition.
First, we introduce some notation.
Let \(\CC{1},\ldots, \CC{m}\) be the components of \(\Gammai\), and let \(\DD{1},\ldots, \DD{n}\) be the components of \(\Gammaii\). 
By Corollary~\ref{cor:pairing-knots:local_systems}, we may assume that all local systems are trivial, so we will treat curves with \(v\)-dimensional local systems as \(v\) distinct components.
By assumption, the tangle \(T_1\) is non-rational, so it follows from Theorem~\ref{thm:HFT:rational_tangle_detection} that $\Gammai$ contains either a special component or multiple rational components. In the first case, there are at least two special components by conjugation symmetry (Theorem~\ref{thm:Conjugation}); in both cases, the number of rational components is odd (Proposition~\ref{prop:Odd_number_of_rationals}). So in either case, \(m\geq3\). Similarly, \(n\geq3\). 
By assumption, \(\DD{i}\) is rational for all \(i=1,\dots,n\) and not all of these components have the same slope.
By Proposition~\ref{prop:Odd_number_of_rationals}, we may further assume without loss of generality that \(\CC{1}\) is rational. 
Finally, let us write \(W_{i,j} = \HFr(\CC{i},\DD{j})\). 

\begin{lemma}\label{lem:contiguity}
	If \(K\) is skeletal,  \(W_{i,j}\) is non-zero and contiguous for any \(i,j\). 
\end{lemma}

\begin{proof}
	Suppose $W_{i,j} = 0$.
	By Corollary~\ref{cor:pairing-knots}, \(\CC{i}\) and \(\DD{j}\) cannot be rational curves of the same slope, since \(K\) is a knot. Therefore, \(\CC{i}\) must be special and must have the same slope as the rational curve \(\DD{j}\). 
	After a reparametrization of the boundary, we may assume that they both have slope 0. This means that $\DD{j}$ separates the tangle ends \(\TEI\) and \(\TEIV\) from \(\TEII\) and \(\TEIII\). Since not all components of $\Gammaii$ have the same slope, there exists some $j_*$ such that $\DD{j_*}$ has slope $\frac{1}{p}$ for some $p \in \mathbb{Z}$, by Lemma~\ref{lem:reduction:mixed}.  This means that $\DD{j_*}$ separates the puncture \(\TEI\) from \(\TEIV\). This contradicts Lemma~\ref{lem:HFT_detects_connectivity}.
	
	If \(\CC{i}\) is rational, 
	\(W_{i,j}\) is \(\HFKhat\) of a knot (up to a grading shift), so contiguity follows from Lemma~\ref{lem:2bridge}.
	If \(\CC{i}\) is special, contiguity follows from naturality under twisting and Lemmas~\ref{lem:reduction:mixed} and~\ref{lem:computation:mixed}.
\end{proof}

By the gluing theorem, 
\[
\HFKhat(K)\cong \HFr(\Gammai,\Gammaii)=\bigoplus W_{i,j}
\]
where the direct sum is over all \(i=1,\dots,n\) and \(j=1,\dots,m\). 
Let us assume from now on that \(K\) is skeletal.
Then by Lemma~\ref{lem:contiguity}, each summand $W_{i,j}$ is non-zero and contiguous, so we can order the contributions from the various \(W_{i,j}\) based on their Alexander gradings.   

\begin{definition}
  We write \((i,j) < (p,q)\) if \(W_{i,j}\) is supported in lower Alexander grading than \(W_{p,q}\). 
\end{definition}

\begin{lemma}\label{lem:ordering}
  For any \(k,\ell\in\{1,\dots,n\}\), \((2,k) < (1,k)\) implies \((2,\ell) < (1,\ell)\).
\end{lemma}
\begin{proof}
	Let us choose a generator \(x^{2}_{k}\in\HF(\CC{2},\DD{k})\) with minimal Alexander grading. After picking an overall absolute Alexander grading of \(\gamma_{1}=\CC{1}\), \(\gamma_{2}=\CC{2}\), \(\vartheta_{1}=\DD{\ell}\), and \(\vartheta_{2}=\DD{k}\), these curves satisfy the hypothesis of Lemma~\ref{lem:Alex:ordering}. 
	Indeed, by assumption \(\CC{1}\), $\DD{k}$, and $\DD{\ell}$ are rational. Hence the slope of \(\CC{1}\) is pairwise different from the slopes of $\DD{k}$ and $\DD{\ell}$ by Lemma~\ref{lem:HFT_detects_connectivity}. If \(\CC{2}\) is rational, then by the same argument the slope of \(\CC{2}\) is pairwise different from the slopes of $\DD{k}$ and $\DD{\ell}$; if \(\CC{2}\) is special, we can argue with Lemma~\ref{lem:contiguity}.
	Therefore, there exist generators
  \[
  x^{1}_{\ell}\in\HF(\CC{1},\DD{\ell}),
  \quad
  x^{2}_{\ell}\in\HF(\CC{2},\DD{\ell}),
  \quad\text{and}\quad
  x^{1}_{k}\in\HF(\CC{1},\DD{k})
  \]
  satisfying
  \[
  A(x^{2}_{k})+A(x^{1}_{\ell})
  =
  A(x^{2}_{\ell})+A(x^{1}_{k}).
  \]
  Now suppose, \((2,k) < (1,k)\). By minimality of \(A(x^{2}_{k})\), it follows that 
  \(A(x^{2}_{k})<A(x^{1}_{k})\), which together with the above identity implies 
  \(A(x^{2}_{\ell})<A(x^{1}_{\ell})\). Since we are assuming that \(W_{2,\ell}\) and \(W_{1,\ell}\) are supported in entirely distinct Alexander gradings, it follows that \((2,\ell) < (1,\ell)\). 
\end{proof} 


\begin{remark}
Lemma~\ref{lem:ordering} is false if \(\CC{1}\) is a special curve: One can use the curves \(\CC{1}=\Special_1(\infty;\TEI,\TEII)\),  \(\CC{2}=\Special_1(\infty;\TEIII,\TEIV)\), \(\DD{1}=\Rational(s)\), and \(\DD{2}=\Rational(-s)\) (for some slope \(s\in\mathbb{Z}\)) to construct a simple family of counterexamples, which may even be skeletal. 
\end{remark} 

\begin{proof}[Proof of Proposition~\ref{prop:NoLspaceKnot:IrrRat}]
	There are some \(k,\ell\) such that \((1,k)<(1,\ell)\). If \((1,k)<(2,k)\), then either
	\[
	\text{(i)}\quad (1,k)<(2,k)<(1,\ell) 
	\qquad\text{or}\qquad
	\text{(ii)}\quad (1,k)<(1,\ell)<(2,k).
	\]
	Conversely, if \((2,k)<(1,k)\), then also \((2,\ell)<(1,\ell)\) by Lemma~\ref{lem:ordering} and we are in one of the following cases:
	\[
	\text{(iii)}\quad (1,k)<(2,\ell)<(1,\ell)
	\qquad\text{or}\qquad
	\text{(iv)}\quad (2,\ell)<(1,k)<(1,\ell).
	\]
	We now set
 \[
 W=
 \begin{cases*}
 \HFr(\CC{1},\Gammaii)
 &
 cases (i) and (iii)
 \\
 \HFr(\Gammai,\DD{k})
 &
 case (ii)
 \\
 \HFr(\Gammai,\DD{\ell})
 &
 case (iv)
 \end{cases*}
 \]
 Now suppose \(K\) is an L-space knot.  But then in each of these cases, the subspace \(W\subset\HFr(\Gammai,\Gammaii)\) violates Corollary~\ref{cor:Lspacepinch}.  
\end{proof}

\newcommand*{\arxiv}[1]{\href{http://arxiv.org/abs/#1}{ArXiv:\ #1}}
\newcommand*{\arxivPreprint}[1]{\href{http://arxiv.org/abs/#1}{ArXiv preprint #1}}

\bibliographystyle{alpha}
\bibliography{0_references}
\end{document}